\documentclass{amsart}%
\usepackage{amsmath}
\usepackage{amssymb}
\usepackage{amsfonts}
\usepackage{geometry}
\usepackage{color}
\usepackage{verbatim}
\usepackage{graphicx}%

\providecommand{\U}[1]{\protect\rule{.1in}{.1in}}

\renewcommand{\hat}{\widehat}

\renewcommand{\epsilon}{\varepsilon}

\newtheorem{theorem}{Theorem}
\theoremstyle{plain}

\newtheorem{corollary}[theorem]{Corollary}

\newtheorem{lemma}[theorem]{Lemma}

\newtheorem{proposition}[theorem]{Proposition}
\newtheorem{remark}[theorem]{Remark}

\numberwithin{equation}{section}
\begin{document}
\title{Haar basis testing}
\author[M. Alexis]{Michel Alexis}
\address{Mathematical Institute, University of Bonn, Endenicher Allee 60, 53115, Bonn, Germany}
\email{micalexis.math@gmail.com}
\author[J. L. Luna-Garcia]{Jose Luis Luna-Garcia}
\address{Department of Mathematics \& Statistics, McMaster University, 1280 Main Street
West, Hamilton, Ontario, Canada L8S 4K1}
\email{lunagaj@mcmaster.ca}
\author[E.T. Sawyer]{Eric T. Sawyer}
\address{Department of Mathematics \& Statistics, McMaster University, 1280 Main Street
West, Hamilton, Ontario, Canada L8S 4K1 }
\email{Sawyer@mcmaster.ca}
\thanks{M. Alexis is funded by the Deutsche Forschungsgemeinschaft (DFG, German
Research Foundation) under Germany's Excellence Strategy - GZ 2047/1,
Projekt-ID 390685813.}
\thanks{E. Sawyer is partially supported by a grant from the National Research Council
of Canada}

\begin{abstract}
We show that for two doubling measures $\sigma$ and $\omega$ on $\mathbb{R}%
^{n}$ and any \emph{fixed} dyadic grid $\mathcal{D}$ in
$\mathbb{R}^{n}$,
\[
\mathfrak{N}_{\mathbf{R}^{\lambda, n}}\left(  \sigma,\omega\right)  \approx\mathfrak{H}%
_{\mathbf{R}^{\lambda, n}}^{\mathcal{D},\operatorname*{glob}}\left(  \sigma,\omega\right)
+\mathfrak{H}_{\mathbf{R}^{\lambda, n}}^{\mathcal{D},\operatorname*{glob}}\left(
\omega, \sigma\right)  \ ,
\]
where $\mathfrak{N}_{\mathbf{R}^{\lambda, n}} (\sigma, \omega)$ denotes the $L^2 (\sigma) \to L^2 (\omega)$
operator norm of the vector-Riesz transform $\mathbf{R}^{\lambda, n}$ of fractional order $\lambda \neq 1$, and%
\[
\mathfrak{H}_{\mathbf{R}^{\lambda,n}}^{\mathcal{D},\operatorname*{glob}}\left(
\sigma,\omega\right)  \equiv\sup_{I\in\mathcal{D}}\left\Vert \mathbf{R}^{\lambda,n} h_{I}^{\sigma}\right\Vert _{L^{2}\left(  \omega\right)  }\ ,
\]
is the global Haar testing characteristic for $\mathbf{R}^{\lambda,n}$ on the grid
$\mathcal{D}$, and $\left\{  h_{I}^{\sigma}\right\}  _{I\in\mathcal{D}}$ is
the weighted Haar orthonormal basis of $L^{2}\left(  \sigma\right)  $ arising in the work
of Nazarov, Treil and Volberg.

We also show this theorem extends more generally to weighted Alpert wavelets which replace the weighted Haar wavelets in the proofs of some recent two-weight $T1$ theorems \cite{RaSaWi, AlSaUr, SaWi}.

Finally, we briefly pose these questions in the context of orthonormal bases
in arbitrary Hilbert spaces.

\end{abstract}
\maketitle
\tableofcontents

\section{Introduction: orthonormal bases}

Suppose $\mathcal{H}_{1}$ and $\mathcal{H}_{2}$ are separable Hilbert spaces
and that $\mathcal{B}\subset B\left(  \mathcal{H}_{1},\mathcal{H}_{2}\right)
$ is a collection of bounded linear operators from $\mathcal{H}_{1}$ to
$\mathcal{H}_{2}$. A natural question is whether the
collection $\mathcal{B}$ is uniformly bounded in $B\left(  \mathcal{H}%
_{1},\mathcal{H}_{2}\right)  $. The \emph{uniform boundedness principle} says
that this is the case \emph{if and only if} the collection is pointwise
bounded, i.e.,
\[
\sup_{T\in\mathcal{B}}\left\Vert Tf\right\Vert _{\mathcal{H}_{2}}<\infty\text{
for every }f\in\mathcal{H}_{1}\ .
\]
This pointwise criterion requires testing over \textbf{all} $f\in
\mathcal{H}_{1}$, or at least over a subset of second category in
$\mathcal{H}_{1}$, and the question of testing over a smaller set of
functions becomes relevant. A minimal set of\ reasonable testing functions is
given by any fixed orthonormal basis $\mathcal{O}_{1}$ of $\mathcal{H}_{1}$,
and since boundedness of $T$ is equivalent to boundedness of the adjoint
$T^{\ast}$, it is reasonable to include testing the adjoint $T^{\ast}$ over
a\ fixed orthonormal basis $\mathcal{O}_{2}$ of $\mathcal{H}_{2}$. This leads
to the following loosely formulated question for a set $\mathcal{B}$ of
bounded linear operators from $\mathcal{H}_{1}$ to $\mathcal{H}_{2}$ and a
pair of orthonormal bases $\mathcal{O}_{1}$ and $\mathcal{O}_{2}$ of
$\mathcal{H}_{1}$ and $\mathcal{H}_{2}$, respectively.
\begin{align}
  \text{Is }\mathcal{B} &\text{ uniformly bounded in }B\left(
\mathcal{H}_{1},\mathcal{H}_{2}\right)  \text{ if }\label{question} T \text{ and } T^{\ast} \text{ are uniformly bounded on } \\&\mathcal{O}_{1}\text{ and } \mathcal{O}_{2}\text{, respectively, for all }T\in\mathcal{B}%
?\nonumber
\end{align}

To be more precise, suppose $\mathcal{H}_{1}$ and $\mathcal{H}_{2}$ are separable Hilbert spaces
with orthonormal bases $\mathcal{O}_{1}=\left(  e_{k}^{1}\right)
_{k\in\mathbb{N}}$ and $\mathcal{O}_{2}=\left(  e_{k}^{2}\right)
_{k\in\mathbb{N}}$ of $\mathcal{H}_{1}$ and $\mathcal{H}_{2}$ respectively.
Define the norm
\[
\left\Vert T\right\Vert _{\mathcal{H}_{1}\rightarrow\mathcal{H}_{2}}\equiv
\sup_{\left\Vert x\right\Vert _{\mathcal{H}_{1}}=1}\left\Vert Tx\right\Vert
_{\mathcal{H}_{2}}
\]
and the testing characteristics
\[
	\left\Vert T\right\Vert _{\mathcal{O}_{1}}%
\equiv\sup_{k\in\mathbb{N}}\left\Vert Te_{k}^{1}\right\Vert _{\mathcal{H}_{2}%
}\text{ and }\left\Vert T^{\ast}\right\Vert _{\mathcal{O}_{2}}\equiv\sup
_{k\in\mathbb{N}}\left\Vert T^{\ast}e_{k}^{2}\right\Vert _{\mathcal{H}_{1}%
}\ ,
\]
where the name ``testing
characteristic'' arises because the supremum involves testing the action of an 
operator on a collection of ``test functions.'' We are interested in finding classes of bounded\ linear operators $\mathcal{B}%
\subset B\left(  \mathcal{H}_{1},\mathcal{H}_{2}\right)  $ from $\mathcal{H}%
_{1}$ to $\mathcal{H}_{2}$ whose norms are controlled by their testing characteristics on $\mathcal{O}_1$ and $\mathcal{O}_2$, i.e., 
\begin{equation}
\left\Vert T\right\Vert _{\mathcal{H}_{1}\rightarrow\mathcal{H}_{2}}\leq
C\left(  \left\Vert T\right\Vert _{\mathcal{O}_{1}}+\left\Vert T^{\ast
}\right\Vert _{\mathcal{O}_{2}}\right)  ,\ \ \ \ \ \text{for every }%
T\in\mathcal{B}. \label{controlled}%
\end{equation}

It is well known that the answer to Question (\ref{question}) can be negative for
$\mathcal{B}=B\left(  \mathcal{H}_{1},\mathcal{H}_{2}\right)  $ and
arbitrarily prescribed orthonormal bases in separable Hilbert
spaces\footnote{The answer to question (\ref{question}) is trivially
affirmative if we are allowed to choose the orthonormal basis $\mathcal{O}%
_{1}$ depending on $\mathcal{B}$. Indeed, simply choose $T\in\mathcal{B}$ and
$x\in\mathcal{H}_{1}$ such that $\frac{\left\Vert Tx\right\Vert _{\mathcal{H}%
_{2}}}{\left\Vert x\right\Vert _{\mathcal{H}_{1}}}>\frac{1}{2}\sup
_{B\in\mathcal{B}}\sup_{0\neq y\in\mathcal{H}_{1}}\frac{\left\Vert
By\right\Vert _{\mathcal{H}_{2}}}{\left\Vert y\right\Vert _{\mathcal{H}_{1}}}$
(if this supremum is infinite, simply choose $\frac{\left\Vert Tx\right\Vert
_{\mathcal{H}_{2}}}{\left\Vert x\right\Vert _{\mathcal{H}_{1}}}$ arbitrarily
large), and then extend $\left\{  x\right\}  $ to an orthonormal basis
$\mathcal{O}_{1}$.}. Indeed, in the\ single space case $\mathcal{H}%
_{1}=\mathcal{H}_{2}=\ell^{2}$, with the standard orthonormal basis $\left(
\mathbf{e}_{n}\right)  _{n=1}^{\infty}$, the question essentially asks
whether an infinite matrix $A=\left[  a_{mn}\right]  _{m,n=1}^{\infty}$
is bounded on $\ell^{2}$ when all of its rows and columns are uniformly in
$\ell^{2}$. For $\gamma>\frac{1}{2}$, the matrix $A$ with $a_{mn}=n^{-\gamma}$
if $n\geq m$ and $0$ otherwise, has all its rows and columns uniformly in
$\ell^{2}$, but the reader can easily compute that for $\frac{1}{2}<\gamma
\leq\frac{3}{4}$, we have $\left\Vert A\mathbf{v}\right\Vert _{\ell^{2}}=\infty$ for
$\mathbf{v}=\left(  n^{-\gamma}\right)  _{n=1}^{\infty}\in\ell^{2}$. Then the
sequence $\mathcal{B}=\left\{  A_{N}\right\}  _{N=1}^{\infty}\subset B\left(
\mathcal{H}_{1},\mathcal{H}_{2}\right)  $ of matrices $A_{N}=\left[
a_{mn}^{N}\right]  _{m,n=1}^{\infty}$, with $a_{mn}^{N}=a_{mn}$ if $n\leq N$
and $0$ otherwise, fails (\ref{controlled}). Hence we ask the following
refinement of Question (\ref{question}).

\begin{description}
\item[Basic question for orthonormal bases in Hilbert spaces] Given separable
Hilbert spaces $\mathcal{H}_{1}$ and $\mathcal{H}_{2}$ and a collection of
bounded linear operators $\mathcal{B}\subset B\left(  \mathcal{H}%
_{1},\mathcal{H}_{2}\right)  $, are there naturally occurring
orthonormal bases $\mathcal{O}_{1}$ and $\mathcal{O}_{2}$ in $\mathcal{H}_{1}$
and $\mathcal{H}_{2}$ such that the answer to question (\ref{question}) is
affirmative, i.e., (\ref{controlled}) holds?
\end{description}

In this paper we will fix a dyadic grid $\mathcal{D}$ in $\mathbb{R}^{n}$, and
first consider the  Hilbert spaces $\mathcal{H}_{1}=L^{2}\left(  \sigma\right)
,\mathcal{H}_{2}=L^{2}\left(  \omega\right)  $ with respective orthonormal
bases $\mathcal{O}_{1}=\left(  h_{I}^{\sigma}\right)  _{I\in\mathcal{D}%
},\mathcal{O}_{2}=\left(  h_{J}^{\omega}\right)  _{J\in\mathcal{D}}$, where
$\sigma$ and $\omega$ are doubling measures on $\mathbb{R}^{n}$, and $\left\{
h_{I}^{\mu}\right\}  _{I\in\mathcal{D}}$ is the collection of classical weighted Haar
wavelets on $\mathbb{R}^{n}$ associated to a measure $\mu$. In our main Theorem \ref{thm:main_L2} in Section \ref{section:Haar_results} below, we show what we call a $Th$ theorem for $T$ equal the 
vector Riesz transform (where the $h$ stands for Haar, and following the natural formulation of the problem in two-weight theory, we integrate the singular integral arising in the operator with respect to $\sigma$). 
Namely, we show that the
class $\mathcal{B}_{\mathbf{R}^{\lambda, n} _\sigma }$ of truncated $\lambda$-fractional
vector Riesz transforms, which is of course
contained in $B\left(  \mathcal{H}_{1},\mathcal{H}_{2}\right)  $, satisfies (\ref{controlled}) when $\lambda \neq 1$. Recall that when $\lambda=0$, this class of
operators includes truncations of the classical Hilbert, Cauchy and Beurling transforms.
This not only gives an affirmative answer to question (\ref{question}) in this
case, it also shows that any vector Riesz transform $\mathbf{R}^{\lambda, n}$ is bounded\footnote{in the sense of Section \ref{Sec CZ}} from $L^{2}\left(  \sigma\right)
$ to $L^{2}\left(  \omega\right)  $  \emph{if and
only if}
\[
\left\Vert  \mathbf{R}^{\lambda, n}\right\Vert _{\mathcal{O}_{1}}+\left\Vert \mathbf{R}^{\lambda, n} \right\Vert
_{\mathcal{O}_{2}}<\infty.
\]

Recall the classical $T1$ theorems\footnote{in which the $1$ in $T1$ represents the indicator of a cube}, which more or less say the norm inequality for certain operators holds if and only if the inequality holds when testing the operator over all indicators of cubes. Compare this to the $Th$ type theorems proposed above: while indicators of cubes are dense in a Hilbert space, the $Th$ theorems have the advantage of only needing to test the operator $T$ on a \emph{countable basis} consisting of Haar wavelets. We point out that a local $Th$ theorem almost trivially implies a global $T1$ theorem: since Haar wavelets are a finite linear combination of indicators of cubes, the Haar testing characteristic is dominated by the (global) cube-testing characteristic by the triangle inequality, and hence any $Th$ theorem implies the (global) $T1$ theorem.

Our main strategy for proving the $Th$ Theorem \ref{thm:main_L2} is to show results like Theorem \ref{main''} which says that for the class of gradient elliptic $\lambda$-fractional operators, the (local) $T1$ theorem implies a (global) $Th$ theorem for doubling measures; we then apply these results to the vector Riesz transform, for which a $T1$ theorem holds \cite{LaWi, SaShUr9, SaWi}. In fact, the $T1$ theorem is the major constraint in our work, in that if one had a stronger $T1$ theorem which held for a wider class of gradient elliptic operators, our work would immediately imply a $Th$ theorem for that very same class of operators.

Theorem \ref{main''} is
interesting because a decomposition of functions into linear
combinations of orthogonal Haar wavelets is vital in the proofs of
the two-weight $T1$ theorems for Calder\'{o}n-Zygmund operators (while there are too many papers to cite, see for instance \cite{NTV,LaSaShUr3, Lac, Hyt2, SaShUr7} and some of the other papers cited in this paper and those ones). 
This suggests  Haar functions play such
a big role in two-weight $T1$ theorems because a $Th$ theorem may also hold: while we can  prove so only for the vector Riesz transform, perhaps for a wide class of Calder\'on-Zygmund operators
$T^{\lambda}$, we have $T^{\lambda}$ satisfies the norm inequality if and only if the norm inequality holds across all Haar wavelets.

Finally we extend the weighted Haar wavelet Theorem \ref{thm:main_L2} to the weighted Alpert wavelet Theorem \ref{main' Alpert} ; see Section \ref{section:Alpert} below for definitions and properties of the Alpert wavelets, where it is noted that the weighted Haar wavelets are simply weighted Alpert wavelets with moment-vanishing parameter $\kappa=1$. Just as for the Haar wavelets, Theorem \ref{main' Alpert} is significant because weighted Alpert wavelets play an analogous role in works on two-weight $T1$ theorems for doubling measures \cite{RaSaWi, AlSaUr, SaWi}.

Similar to the Haar case, the (local) Alpert testing characteristic is also controlled by the (global) cube-testing characteristic, and hence any ``Alpert testing theorem'' implies the $T1$ theorem. Indeed, first approximate an Alpert wavelet by a finite linear combination of indicators of cubes which are at least a fixed proportion of the size of the support of the wavelet. By the triangle inequality, the Alpert testing characteristic is then bounded by the cube-testing characteristic and a small multiple of the norm, which can be then absorbed. See \cite[proof of Theorem 20]{AlSaUr} for an analogous argument.

\section{Preliminaries: Calder\'on-Zygmund operators and weighted Haar wavelets}

\subsection{The operators: $\lambda$-fractional Calder\'{o}n-Zygmund
operators\label{Sec CZ}}

The class of bounded operators we consider are the
admissible truncations of smooth $\lambda$-fractional Calder\'{o}n-Zygmund
operators $T^{\lambda}$ on $\mathbb{R}^{n}$ for $0\leq\lambda<n$, as defined
for example in \cite[Section 2.1.2]{AlSaUr}, \cite{SaShUr9}. Such operators are associated
with real, smooth kernels $K^{\lambda}\left(  x,y\right)  $ satisfying the Calder\'on-Zygmund size and smoothness estimates
\begin{align} \label{size and smoothness}
  \left\vert \nabla_{x}^{m}K^{\lambda}\left(  x,y\right)  \right\vert
+\left\vert \nabla_{y}^{m}K^{\lambda}\left(  x,y\right)  \right\vert \leq
C_{\operatorname*{CZ}}\left\vert x-y\right\vert ^{\lambda-n-m} \, , \quad \text{for }m\geq0,\ \text{and }x,y\in\mathbb{R}^{n}\text{
with}\ x\neq y \, .
\end{align}
In particular, given a kernel $K^{\lambda}$, we may define smooth
\emph{admissible} truncations $K_{\varepsilon,R}^{\lambda}\left(  x,y\right)
$ which are constructed so that they vanish on $\{(x,y)\in\mathbb{R}%
^{n}\times\mathbb{R}^{n}~:~|x-y|<\varepsilon\text{ or }|x-y|>R\}$, are smooth,
and continue to satisfy (\ref{size and smoothness}). We identify
the collection of admissible truncations $K_{\varepsilon,R}^{\lambda}\left(
x,y\right)  $ of the kernel $K^{\lambda}$ with the operator $T^{\lambda}$, and we say that $T^{\lambda}$ is bounded from
$L^{p}\left(  \sigma\right)  $ to $L^{p}\left(  \omega\right)  $ if there exists a constant $\mathfrak{N}_{T^{\lambda},p}\left(  \sigma,\omega\right)$, referred to as the norm of $T^{\lambda}$, such that for all $f \in L^p \left (\sigma \right )$, we have
\begin{align*}
\left(  \int_{\mathbb{R}^{n}}\left\vert \int_{\mathbb{R}^{n}}K_{\varepsilon
,R}^{\lambda}\left(  x,y\right)  f\left(  y\right)  d\sigma\left(  y\right)
\right\vert ^{p}d\omega\left(  x\right)  \right)  ^{\frac{1}{p}} &
\leq\mathfrak{N}_{T^{\lambda},p}\left(  \sigma,\omega\right)  \left(
\int_{\mathbb{R}^{n}}\left\vert f\left(  y\right)  \right\vert ^{p}%
d\sigma\left(  y\right)  \right)  ^{\frac{1}{p}}\, .
\end{align*}

See, e.g., \cite{SaShUr9} for more details on this, and see \cite{Saw6} for how such bounds imply the
existence of a bounded operator that agrees with the truncations, extending the Lebesgue measure case in \cite{Ste2}.

Let $\kappa$ be a nonnegative integer. We say that $T^{\lambda}$ is a $\kappa
$\emph{-elliptic} Calder\'{o}n-Zygmund operator on $\mathbb{R}^{n}$ if in
addition to (\ref{size and smoothness}), there is a unit vector $\mathbf{v}%
\in\mathbb{S}^{n-1}$ and $c_{1}>0$ such that the kernel $K^{\lambda}\left(
x,y\right)  $ of $T^{\lambda}$ satisfies,
\begin{equation}
\label{eq:gradient_elliptic}\left\vert \left(
\frac{d}{dt}\right)  ^{\kappa}K^{\lambda}\left(  x,x+t\mathbf{v}\right)
\right\vert \geq c_{1}\left | t \right |^{\lambda-n-\kappa},\ \ \ \ \ \text{for all }%
x\in\mathbb{R}^{n}\text{ and }t>0\text{.}%
\end{equation}
Note that by the change of variables $y = x+ t \mathbf{v}$ and replacing $t$ by $-t$, we have \eqref{eq:gradient_elliptic} is equivalent to  
\[
\left\vert \left(  \frac{d}{dt}\right)  ^{\kappa
}K^{\lambda}\left(  y+t\mathbf{v},y\right)  \right\vert \geq c_{1}\left | t \right |^{\lambda-n-\kappa} \, , \quad \text{for all }%
y \in\mathbb{R}^{n}\text{ and }t>0\text{.}%
\, . 
\]
If $\kappa=0$, then $T^{\lambda}$ is called \emph{Stein elliptic}. If $\kappa=1$, then $T^{\lambda}$ is called \emph{gradient elliptic}, as introduced in \cite{SaShUr10}. One easily checks
that in the case $\kappa=1$, both functions
\[
\varphi_{y,\mathbf{v}}\left(  t\right)  =K^{\lambda}\left(  y+t\mathbf{v}%
,y\right)  \text{ and }\psi_{x,\mathbf{v}}\left(  t\right)  =K^{\lambda
}\left(  x,x+t\mathbf{v}\right)
\]
are monotone on $\left(  0,\infty\right)  $ and $\left(  -\infty,0\right)  $,
and that
\[
\left\vert \varphi_{y,\mathbf{v}}^{\prime}\left(  t\right)  \right\vert
\approx\left\vert \psi_{y,\mathbf{v}}^{\prime}\left(  t\right)  \right\vert
\approx\left\vert t\right\vert ^{\lambda-n-1},\ \ \ \ \ \text{for all }%
x,y\in\mathbb{R}^{n}\,\ \text{and }t\in\mathbb{R\setminus}\left\{  0\right\}
.
\]
We point out that $\kappa$-ellipticity implies $\kappa^{\prime}%
$-ellipticity for all $0\leq\kappa^{\prime}\leq\kappa$ upon using the
fundamental theorem of calculus on infinite rays, together with the size and
smoothness bounds (\ref{size and smoothness}).

We will often make reference to the \emph{Calder\'on-Zygmund data} for an operator $T^{\lambda}$: this refers to the constants in (\ref{size and smoothness}) and, depending on the context, the constant in \eqref{eq:gradient_elliptic}. 

Our main operator of interest is the vector Riesz transform. Given $0 \leq \lambda < n$, define the $\lambda$-fractional Riesz transform $\mathbf{R}^{\lambda, n}$ by
\[
\mathbf{R}^{\lambda, n}f (x) := \operatorname{p.v.} \int\limits_{\mathbb{R}^n} \frac{x-y}{|x-y|^{n+1 - \lambda}} f(y) \, d y \, .
\]
This vector-valued operator is $\kappa$-elliptic for all $\kappa \geq 0$.

\subsection{Orthonormal bases and weighted $L^{2}$ spaces}

For us, a cube will always be axis-parallel unless otherwise specified. We let $\mathcal{P}^{n}$ denote the set of all axis-parallel cubes in $\mathbb{R}^{n}$.

Recall that a measure $\mu$ on $\mathbb{R}^{n}$ is doubling if there exists a
constant $C>0$ such that 
\begin{equation}\label{eq:doubling}
	\mu\left(  2 Q \right)  \leq C \mu\left(  Q \right) \quad \text{ for all cubes } Q \, ,
\end{equation}
where $2 Q$ is the cube with same center but with sidelength double that of $Q$. The best constant $C$ in \eqref{eq:doubling} is the
\emph{doubling constant} for $\mu$.

A tiling $\mathcal{T}$ of $\mathbb{R}^{n}$ is a collection
of cubes $Q$ whose interiors are pairwise disjoint, and whose union is
$\mathbb{R}^{n}$. A dyadic grid $\mathcal{D}$ is a union
$\bigcup\limits_{k \in\mathbb{Z}} \mathcal{D}_{k}$ of tilings $\mathcal{D}%
_{k}$ of $\mathbb{R}^{n}$ satisfy the
following additional conditions

\begin{itemize}
	\item all cubes in $\mathcal{D}_{k}$ have sidelength $2^{-k}$.

\item any cube $Q \in\mathcal{D}_{k+1}$ is contained in a cube $Q^{\prime}%
\in\mathcal{D}_{k}$, called the dyadic parent of $Q$.

\item every cube $Q^{\prime}\in\mathcal{D}_{k}$ contains  $2^{n}$ cubes 
in $\mathcal{D}_{k+1}$, called its dyadic children.
\end{itemize}

The standard dyadic grid on $\mathbb{R}$ is the collection
of intervals $\left[  j2^{-k},\left(  j+1\right)  2^{-k}\right)  $ for
$j,k\in\mathbb{Z}$. Given a measure $\mu$ on $\mathbb{R}^{n}$ and an interval
$I$, a $\mu$-weighted Haar wavelet $h_{I}^{\mu}$ on $I$ is an $L^2 \left (\mu \right )$-normalized 
function satisfying the following properties:

\begin{itemize}
\item $\operatorname*{Supp}h_{I}^{\mu}$ is a union of dyadic children of $I$.

\item $h_{I}^{\mu}$ is constant on each dyadic child of $I$.

\item $\int h_{I}^{\mu}d\mu=0$.
\end{itemize}

If $I \subset \mathbb{R}^n$, the linear span of all the $\mu
$-weighted Haar wavelets on $I$ has dimension $2^{n}-1$. Hence when $n=1$, a Haar
wavelet $h_{I}^{\mu}$ on $I$ is determined uniquely up to its sign. But when $n \geq 2$, there many choices of $L^{2}(\mu
)$-orthonormal vectors $\{h_{I}^{\mu}\}$ which span this space. We let $\left\{  h_{I}^{\mu,\gamma
}\right\}  _{\gamma\in\Gamma_{I,n}^{\mu}}$ denote a choice of orthonormal
basis for this space, where $\Gamma_{I,n}^{\mu}$ is an index set of size equal to the dimension $2^{n}-1$.
But often we will be imprecise about which specific Haar wavelet we are considering, and will simply write
$h_{I}^{\mu}$ to refer to \emph{some} $\mu$-weighted Haar wavelet on 
$I$; similarly $\left\{  h_{I}^{\mu}\right\}  $ refers to the
collection of $\mu$-weighted Haar wavelets on $I$.

If a measure $\mu$ is doubling, then $\left\{  h_{I}^{\mu, \gamma}\right\}
_{I\in\mathcal{D}, \gamma \in \Gamma_{I, n} ^{\mu}}$ is an orthonormal basis of $L^{2}\left(  \mu\right)  $.
Given a dyadic grid $\mathcal{D}$,$\ $define the global and local
$\mathcal{D}$-Haar testing characteristics for $T^{\lambda}$ by%
\begin{align*}
\mathfrak{H}_{T^{\lambda}}^{\mathcal{D},\operatorname*{glob}}\left(
\sigma,\omega\right)   &  \equiv\sup_{I\in\mathcal{D}} \sup\limits_{ \left\{
h_{I}^{\sigma}\right\}  } \left\Vert T_{\sigma}^{\lambda}h_{I}^{\sigma
}\right\Vert _{L^{2}\left(  \omega\right)  }\text{ and }\mathfrak{H}%
_{T^{\lambda}}^{\mathcal{D},\operatorname*{glob}}\left(  \omega,\sigma\right)
\equiv\sup_{I\in\mathcal{D}} \sup\limits_{\left\{  h_{I}^{\omega}\right\}
}\left\Vert T_{\omega}^{\lambda,\ast}h_{I}^{\omega}\right\Vert _{L^{2}\left(
\sigma\right)  },\\
\mathfrak{H}_{T^{\lambda}}^{\mathcal{D},\operatorname{loc}}\left(
\sigma,\omega\right)   &  \equiv\sup_{I\in\mathcal{D}} \sup\limits_{ \left\{
h_{I}^{\sigma}\right\}  }\left\Vert \mathbf{1}_{I}T_{\sigma}^{\lambda}%
h_{I}^{\sigma}\right\Vert _{L^{2}\left(  \omega\right)  }\text{ and
}\mathfrak{H}_{T^{\lambda}}^{\mathcal{D},\operatorname{loc}}\left(
\omega,\sigma\right)  \equiv\sup_{I\in\mathcal{D}} \sup\limits_{ \left\{
h_{I}^{\omega}\right\}  }\left\Vert \mathbf{1}_{I}T_{\omega}^{\lambda}%
h_{I}^{\omega}\right\Vert _{L^{2}\left(  \sigma\right)  },
\end{align*}
and similarly for the familiar global, triple and local cube testing
characteristics for $T^{\lambda}$,%
\begin{align*}
\mathfrak{T}_{T^{\lambda}}^{\mathcal{D},\operatorname*{glob}}\left(
\sigma,\omega\right)   &  \equiv\sup_{I\in\mathcal{D}}\left\Vert T_{\sigma
}^{\lambda}\frac{\mathbf{1}_{I}}{\sqrt{\left\vert I\right\vert _{\sigma}}%
}\right\Vert _{L^{2}\left(  \omega\right)  }\text{ and }\mathfrak{T}%
_{T^{\lambda}}^{\mathcal{D},\operatorname*{glob}}\left(  \omega,\sigma\right)
\equiv\sup_{I\in\mathcal{D}}\left\Vert T_{\omega}^{\lambda,\ast}%
\frac{\mathbf{1}_{I}}{\sqrt{\left\vert I\right\vert _{\omega}}}\right\Vert
_{L^{2}\left(  \sigma\right)  },\\
\mathfrak{T}_{T^{\lambda}}^{\mathcal{D},\operatorname*{trip}}\left(
\sigma,\omega\right)   &  \equiv\sup_{I\in\mathcal{D}}\left\Vert \mathbf{1}_{3 I} T_{\sigma
}^{\lambda}\frac{\mathbf{1}_{I}}{\sqrt{\left\vert I\right\vert _{\sigma}}%
}\right\Vert _{L^{2}\left(  \omega\right)  }\text{ and }\mathfrak{T}%
_{T^{\lambda}}^{\mathcal{D},\operatorname*{trip}}\left(  \omega,\sigma\right)
\equiv\sup_{I\in\mathcal{D}}\left\Vert \mathbf{1}_{3 I} T_{\omega}^{\lambda,\ast}%
\frac{\mathbf{1}_{I}}{\sqrt{\left\vert I\right\vert _{\omega}}}\right\Vert
_{L^{2}\left(  \sigma\right)  },\\
\mathfrak{T}_{T^{\lambda}}^{\mathcal{D},\operatorname{loc}}\left(
\sigma,\omega\right)   &  \equiv\sup_{I\in\mathcal{D}}\left\Vert
\mathbf{1}_{I}T_{\sigma}^{\lambda}\frac{\mathbf{1}_{I}}{\sqrt{\left\vert
I\right\vert _{\sigma}}}\right\Vert _{L^{2}\left(  \omega\right)  }\text{ and
}\mathfrak{T}_{T^{\lambda}}^{\mathcal{D},\operatorname{loc}}\left(
\omega,\sigma\right)  \equiv\sup_{I\in\mathcal{D}}\left\Vert \mathbf{1}%
_{I}T_{\omega}^{\lambda,\ast}\frac{\mathbf{1}_{I}}{\sqrt{\left\vert
I\right\vert _{\omega}}}\right\Vert _{L^{2}\left(  \sigma\right)  }.
\end{align*}

We will also need the local cube testing characteristics, defined as
\[
\mathfrak{T}_{T^{\lambda}}^{\operatorname{loc}}\left(  \sigma,\omega\right)
\equiv\sup_{I\in\mathcal{P}^{n}}\left\Vert \mathbf{1}_{I}T_{\sigma}^{\lambda
}\frac{\mathbf{1}_{I}}{\sqrt{\left\vert I\right\vert _{\sigma}}}\right\Vert
_{L^{2}\left(  \omega\right)  }\text{ and }\mathfrak{T}_{T^{\lambda}%
}^{\operatorname{loc}}\left(  \omega,\sigma\right)  \equiv\sup_{I\in
\mathcal{P}^{n}}\left\Vert \mathbf{1}_{I}T_{\omega}^{\lambda,\ast}%
\frac{\mathbf{1}_{I}}{\sqrt{\left\vert I\right\vert _{\omega}}}\right\Vert
_{L^{2}\left(  \sigma\right)  }\, .
\]

Finally, we define the fractional Muckenhoupt characteristic%
\begin{equation}
\label{eq:A2}A_{2}^{\lambda}\left(  \sigma,\omega\right)  \equiv\sqrt
{\sup_{I\in\mathcal{P}^{n}}\frac{\left\vert I\right\vert _{\sigma}}{\left\vert
I\right\vert ^{1-\frac{\lambda}{n}}}\frac{\left\vert I\right\vert _{\omega}%
}{\left\vert I\right\vert ^{1-\frac{\lambda}{n}}}},\ \ \ \ \ 0\leq\lambda<n \, .
\end{equation}

\section{Main results: Haar wavelets}\label{section:Haar_results}
If $f$ and $g$ are two nonnegative functions, we write $f \lesssim g$ if there exists a positive constant $C$ such that
\[
f \leq C g \, ,
\]
for all instances of the arguments of $f, g$. We define $f \gtrsim g$ similarly, and write $f \approx g$ if $f \lesssim g$ and $f \gtrsim g$.

Throughout this paper, the  constants implicit  in $\lesssim, \gtrsim$ and
$\approx$ only depend on the doubling constants for
$\sigma,\omega$, the Calder\'{o}n-Zygmund data for the 
operator $T^{\lambda}$, $n$, and also on $\kappa$ and $p$ when they appear. Our main theorem below only concerns the vector Riesz transform.
\begin{theorem}
\label{thm:main_L2} \label{main'}If $\sigma$ and $\omega$ are doubling
measures on $\mathbb{R}^{n}$, $\mathcal{D}$ is a dyadic grid on $\mathbb{R}%
^{n}$, $\lambda \in [0,n) \setminus \{1\}$, then%
\[
\mathfrak{N}_{\mathbf{R}^{\lambda,n}}\left(  \sigma,\omega\right)  \approx\mathfrak{H}%
_{\mathbf{R}^{\lambda,n}}^{\mathcal{D},\operatorname*{glob}}\left(  \sigma,\omega\right)
+\mathfrak{H}_{\mathbf{R}^{\lambda,n}}^{\mathcal{D},\operatorname*{glob}}\left(
\omega,\sigma\right)  \ .
\]

\end{theorem}

This theorem follows immediately from the following two results, the former of which is a special case of the main theorem from \cite{LaWi} (see also \cite{SaShUr9}), and the latter of which is the main focus of this paper.

\begin{theorem}
[\cite{LaWi, SaShUr9}]\label{T1_L2_Riesz} Assume that $\sigma$ and $\omega$ are doubling measures on $\mathbb{R}^{n}$, and let $0 \leq \lambda < n$. If $\lambda \neq 1$, then a $T1$ theorem holds for $\mathbf{R}^{\lambda,n}$, i.e., \begin{equation}
\mathfrak{N}_{\mathbf{R}^{\lambda,n}}\left(  \sigma,\omega\right)  \approx   A_{2}^{\lambda}\left(  \sigma,\omega\right)  +\mathfrak{T}%
_{\mathbf{R}^{\lambda,n}}^{\operatorname{loc}}\left(  \sigma,\omega\right)
+\mathfrak{T}_{\mathbf{R}^{\lambda,n}}^{\operatorname{loc}%
}\left(  \omega,\sigma\right)    \label{bound} \,  .
\end{equation}
\end{theorem}

\begin{theorem}
\label{main''}Let $\sigma$ and $\omega$ be doubling
measures on $\mathbb{R}^{n}$, let $\mathcal{D}$ is a dyadic grid on $\mathbb{R}%
^{n}$, and let $T^{\lambda}$ be a gradient elliptic Calder\'on-Zygmund operator of fractional order $\lambda$. If
\[
\mathfrak{N}_{T^{\lambda}}\left(  \sigma,\omega\right)  \approx A_2 ^{\lambda} (\sigma, \omega)+ \mathfrak{T}%
_{T^{\lambda}}^{\operatorname*{loc}}\left(  \sigma,\omega\right)
+\mathfrak{T}_{T^{\lambda}}^{\operatorname*{loc}}\left(
\omega,\sigma\right)
\]then%
\[
\mathfrak{N}_{T^{\lambda}}\left(  \sigma,\omega\right)  \approx\mathfrak{H}%
_{T^{\lambda}}^{\mathcal{D},\operatorname*{glob}}\left(  \sigma,\omega\right)
+\mathfrak{H}_{T^{\lambda}}^{\mathcal{D},\operatorname*{glob}}\left(
\omega,\sigma\right)  \ .
\]
\end{theorem}

The reason that our main Theorem only concerns the fractional Riesz transform is that this is one of the few operators for which a $T1$ theorem like Theorem \ref{T1_L2_Riesz} is known.

To prove Theorem \ref{main''}, we begin with the following result, whose proof encapsulates the key new ideas of this paper. Note that this result holds for a much wider class of operators than the vector Riesz transforms.

\begin{theorem}
\label{Haar}If $\sigma$ and $\omega$ are doubling measures on $\mathbb{R}^{n}%
$, $\mathcal{D}$ is a dyadic grid on $\mathbb{R}^{n}$, $0\leq\lambda<n$, and
$T^{\lambda}$ is a gradient elliptic $\lambda$-fractional Calder\'{o}n-Zygmund
operator on $\mathbb{R}^n$, then%
\[
\mathfrak{T}_{T^{\lambda}}^{\mathcal{D},\operatorname*{trip}}\left(
\sigma,\omega\right)  +A_{2}^{\lambda}\left(  \sigma,\omega\right)
\lesssim\mathfrak{H}_{T^{\lambda}}^{\mathcal{D},\operatorname*{glob}}\left(
\sigma,\omega\right)  \ .
\]

\end{theorem}

\begin{proof}
We first consider the $A_{2}^{\lambda}\left(  \sigma,\omega\right)  $
characteristic. Let $\mathbf{v}$ be a unit vector for which the
gradient ellipticity condition (\ref{eq:gradient_elliptic}) holds for our
operator $T^{\lambda}$. Then there exists $\delta_{0}>0$ which only depends on the gradient ellipticity constant in \eqref{eq:gradient_elliptic}, such that
\begin{equation}
\left\vert \frac{\partial}{\partial t}K^{\lambda}\left(  x,x+t\mathbf{w}%
\right)  \right\vert =\left\vert \psi_{x,\mathbf{w}}^{\prime}\left(  t\right)
\right\vert \approx\left\vert t\right\vert ^{\lambda-n-1},\ \ \ \ \ \text{for
all}\left\vert \mathbf{w}-\mathbf{v}\right\vert <\delta_{0}.\label{perturb}%
\end{equation}
Indeed, first we have
\begin{align*}
&  \frac{\partial}{\partial t}K^{\lambda}\left(  x,x+t\mathbf{w}\right)
-\frac{\partial}{\partial t}K^{\lambda}\left(  x,x+t\mathbf{v}\right)
=\left(  \mathbf{w}\cdot\nabla_{2}\right)  K^{\lambda}\left(  x,x+t\mathbf{w}%
\right)  -\left(  \mathbf{v}\cdot\nabla_{2}\right)  K^{\lambda}\left(
x,x+t\mathbf{v}\right)  \\
&  =\left(  \left(  \mathbf{w}-\mathbf{v}\right)  \cdot\nabla_{2}\right)
K^{\lambda}\left(  x,x+t\mathbf{w}\right)  +\left(  \mathbf{v}\cdot\nabla
_{1}\right)  \left(  K^{\lambda}\left(  x,x+t\mathbf{w}\right)  -K^{\lambda
}\left(  x,x+t\mathbf{v}\right)  \right)  .
\end{align*}
By the Calder\'{o}n-Zygmund size and smoothness conditions 
(\ref{size and smoothness}), we have
\begin{align*}
\left\vert \left(  \left(  \mathbf{w}-\mathbf{v}\right)  \cdot\nabla
_{2}\right)  K^{\lambda}\left(  x,x+t\mathbf{w}\right)  \right\vert  &  \leq
C_{\operatorname*{CZ}}\left\vert \mathbf{w}-\mathbf{v}\right\vert \left\vert
t\right\vert ^{\lambda-n-1},\\
\left\vert \left(  \mathbf{v}\cdot\nabla_{1}\right)  \left(  K^{\lambda
}\left(  x,x+t\mathbf{w}\right)  -K^{\lambda}\left(  x,x+t\mathbf{v}\right)
\right)  \right\vert  &  \leq C_{\operatorname*{CZ}}\left\vert t\mathbf{w}%
-t\mathbf{v}\right\vert \left\vert t\right\vert ^{\lambda-n-2}.
\end{align*}
Thus by gradient ellipticity, we get
\[
\left\vert \frac{\partial}{\partial t}K^{\lambda}\left(  x,x+t\mathbf{w}%
\right)  -\frac{\partial}{\partial t}K^{\lambda}\left(  x,x+t\mathbf{v}%
\right)  \right\vert \leq C_{\operatorname*{CZ}}\left\vert \mathbf{w}%
-\mathbf{v}\right\vert \left\vert t\right\vert ^{\lambda-n-1}\leq\frac{1}%
{2}\left\vert \frac{\partial}{\partial t}K^{\lambda}\left(  x,x+t\mathbf{v}%
\right)  \right\vert ,
\]
for $\left\vert \mathbf{w}-\mathbf{v}\right\vert <\delta_{0}$ with $\delta
_{0}>0$ sufficiently small, and (\ref{perturb}) follows.

Now given $\delta>0$, let $S\left(  \mathbf{v},\delta\right)  $ denote the
conical sector
\[
S\left(  \mathbf{v},\delta\right)  \equiv\left\{  z\in\mathbb{R}%
^{n}:\left\vert \frac{z}{\left\vert z\right\vert }-\mathbf{v}\right\vert
<\delta\right\}  \text{.}%
\]
Let $0<\delta\leq\delta_{0}$ be a small constant that we will fix later. Then
there is $m=m(\delta)\in\mathbb{N}$ with the following property. Given dyadic
cubes $I$ and $J$ of equal side length $\ell\left(  I\right)  $ with
$\operatorname*{dist}\left(  I,J\right)  \approx\frac{1}{\delta}\ell\left(
I\right)  $, and centers $c_{I}$ and $c_{J}$, such that $J\subset
c_{I}+S\left(  \mathbf{v},\delta\right)  $, then there exist two dyadic cubes
$K,L\in\mathfrak{C}_{\mathcal{D}}^{\left(  m\right)  }\left(  I\right)  $ such
that $\operatorname*{dist}\left(  3K,3L\right)  \approx\ell\left(  I\right)  $
and $L\subset c_{K}+S\left(  \mathbf{v},\delta\right)  $. See Figure
\ref{cubes}.%
\begin{figure}[ht] 
  \fbox{\includegraphics[width=0.75\linewidth]{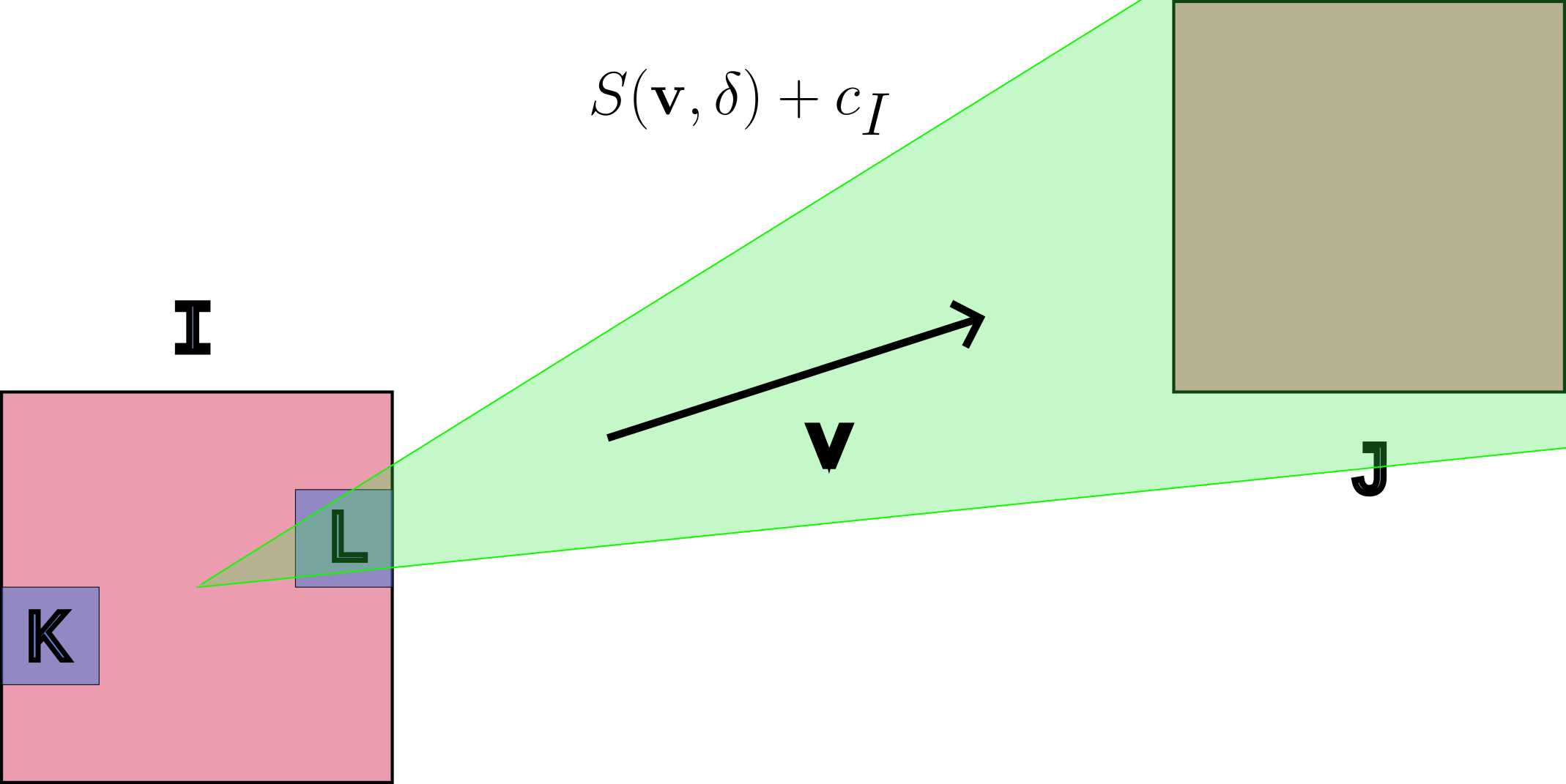}}
	\caption{Cubes $I, J$ (and $K,L$) configured along the conical sector $S(\mathbf{v}, \delta) + c_I$ (or $S(\mathbf{v}, \delta) + c_K$). }
\label{cubes}
\end{figure}
Set $c\equiv\frac{c_{K}+c_{L}}{2}$ and consider
\begin{equation}
\varphi\equiv\frac{1}{\left\vert L\right\vert _{\sigma}}\mathbf{1}_{L}%
-\frac{1}{\left\vert K\right\vert _{\sigma}}\mathbf{1}_{K} \, .\label{def phi}
\end{equation}
Since $\varphi$ has $\sigma$-mean zero, i.e., $\int\varphi\,d\sigma=0$, then for $x\in J$ we have%
\begin{align}
&  T_{\sigma}^{\lambda}\varphi\left(  x\right)  =\int_{I}K^{\lambda}\left(
x,y\right)  \left(  \frac{1}{\left\vert L\right\vert _{\sigma}}\mathbf{1}%
_{L}\left(  y\right)  -\frac{1}{\left\vert K\right\vert _{\sigma}}%
\mathbf{1}_{K}\left(  y\right)  \right)  d\sigma\left(  y\right)
\label{T phi pointwise}\\
&  =\int_{I}\left(  K^{\lambda}\left(  x,y\right)  -K^{\lambda}\left(
x,c\right)  \right)  \left(  \frac{1}{\left\vert L\right\vert _{\sigma}%
}\mathbf{1}_{L}\left(  y\right)  -\frac{1}{\left\vert K\right\vert _{\sigma}%
}\mathbf{1}_{K}\left(  y\right)  \right)  d\sigma\left(  y\right)  \, .\nonumber
\end{align}
We may assume
\begin{equation}
\label{eq:assumption}\left(  \left(  x-c \right)  \cdot\nabla_{2}\right)
K^{\lambda}\left(  x,c\right)  >0
\end{equation}
by (\ref{perturb}) since $x-c\in S\left(  \mathbf{v},\delta_{0}\right)  $ for
$\delta= \delta(\delta_{0})$ sufficiently small (the case  $<0$ is similar). Then we claim that for $x \in J$ and $y
\in L \cup K$, we have
\begin{equation}
\operatorname{sgn} \left\{  K^{\lambda}\left(  x,y\right)  -K^{\lambda}\left(
x,c\right)  \right\}  = \operatorname{sgn}\left(  y-c\right)  \cdot\mathbf{v}
\, , \quad\text{ and } \left|  K^{\lambda}\left(  x,y\right)  -K^{\lambda
}\left(  x,c\right)  \right|  \approx\frac{1}{\left\vert I\right\vert
^{1-\frac{\lambda}{n}}}. \label{sign}%
\end{equation}

To see this, first note that $\left\vert \frac{y-c}{\left\vert y-c\right\vert
}-\mathbf{v}\right\vert <C\delta$ when $y\in L$, and $\left\vert \frac
{y-c}{\left\vert y-c\right\vert }+\mathbf{v}\right\vert <C\delta$ when $y\in
K$. Note then that $\left\vert \frac{y-c}{\left\vert y-c\right\vert }-\left(
\pm\mathbf{v}\right)  \right\vert <C\delta$ implies that $\operatorname{sgn}%
(y-c)\cdot\mathbf{v}=\pm1$ for $\delta$ sufficiently small, meaning that
\[
\operatorname{sgn}(y-c)\cdot\mathbf{v}=%
\begin{cases}
+1 & \text{ if }y\in L\\
-1 & \text{ if }y\in K
\end{cases}
\,.
\]
Now let $\mathbf{w}\equiv\frac{x-c}{\left\vert x-c\right\vert }$. If
$y\in L$ and $x\in J$, then $\left\vert \frac{y-c}{\left\vert
y-c\right\vert }-\mathbf{w}\right\vert <C\delta$. Under these assumptions, we
compute
\begin{align*}
	& K^{\lambda}\left(  x,y\right)  -K^{\lambda}\left(  x,c\right) =\int_{0}^{1}\frac{\partial}{\partial t}K^{\lambda}\left(  x,c  +t\left(  y-c\right)  \right)  dt =\left\vert y-c\right\vert \int_{0}^{1}\left[  \left(  \frac{y-c}{\left\vert y-c\right\vert }\right)  \cdot\nabla_{2}K^{\lambda}\left(
x,c  +t\left(  y-c\right)  \right)  \right]  dt\\
	=&\left\vert y-c\right\vert \int_{0}^{1}\left[  \left(  \frac{y-c}%
{\left\vert y-c\right\vert }-\mathbf{w}\right)  \cdot\nabla_{2}K^{\lambda
}\left(  x,c  +t\left(  y-c\right)  \right)  \right]  dt\\
&  +\left\vert y-c\right\vert \int_{0}^{1}\mathbf{w}\cdot\left[  \nabla
_{2}K^{\lambda}\left(  x,c  +t\left(  y-c\right)  \right)
-\nabla_{2}K^{\lambda}\left(  x,c\right)  \right]  dt  +\left\vert y-c\right\vert \int_{0}^{1}\left[  \mathbf{w}\cdot\nabla
_{2}K^{\lambda}\left(  x,c\right)  \right]  dt\\
&  \equiv i+ii+iii\,.
\end{align*}
We show that $iii$ is the main term, while $i$ and $i i$ are negligible
errors if $\delta$ sufficiently small. Using
Cauchy-Schwarz and the Calder\'{o}n-Zygmund size and smoothness estimates
(\ref{size and smoothness}), we see $\left | i \right|$ is at most
\[
\left\vert y-c\right\vert \left\vert \int_{0}%
^{1}\left[  \left(  \frac{y-c}{\left\vert y-c\right\vert }-\mathbf{w}\right)
\cdot\nabla_{2}K^{\lambda}\left(  x,c  +t\left(
y-c\right)  \right)  \right]  dt\right\vert \leq CC_{\operatorname*{CZ}}%
\ell\left(  I\right)  \delta \int\limits_{0} ^{1} \frac{1}{\left\vert \left(  c-x\right)  +t\left(
y-c\right)  \right\vert ^{n-\lambda+1}} \,  dt \, ,
\]
Because $t\in\lbrack0,1]$ and $|y-c|$ is small relative to $ |x-c|$, then 
\begin{equation}\label{eq:comp_int}
	\left\vert \left(
c-x\right)  +t\left(  y-c\right)  \right\vert \approx\left\vert c-x\right\vert
\approx\frac{1}{\delta}\ell\left(  I\right) \, , 
\end{equation}
and so 
\[
\left\vert i\right\vert \leq CC_{\operatorname*{CZ}}\delta^{n-\lambda+2}%
\frac{1}{\ell\left(  I\right)  ^{n-\lambda}}\,.
\]
To estimate $ii$, we compute
\begin{align*}
	&  \left\vert \nabla_{2}K^{\lambda}\left(  x,c+t\left ( y-c \right )
\right)  -\nabla_{2}K^{\lambda}\left(  x,c\right)  \right\vert =\left\vert
	\int_{0}^{t}\frac{d}{ds}\nabla_{2}K^{\lambda}\left(  x,c+s \left (y-c\right)  \right)  ds\right\vert \\
&  \leq\left\vert y-c\right\vert \int_{0}^{t}\left\vert \nabla_{2}%
^{2}K^{\lambda}\left(  x,c+s \left (y-c\right)  \right)  \right\vert
ds\leq C_{\operatorname*{CZ}}\left\vert y-c\right\vert \int_{0}^{t}\frac
{ds}{\left\vert  c-x +s \left (y-c\right)\right\vert ^{n-\lambda+2}}\leq CC_{\operatorname*{CZ}%
}\ell\left(  I\right)  \frac{\delta^{n-\lambda+2}}{\ell\left(  I\right)
^{n-\lambda+2}}\,
\end{align*}
where in the last step we used \eqref{eq:comp_int} with $t$ replaced by $s$. Hence $|ii|$ is at most 
\begin{align*}
 \left\vert y-c\right\vert \left\vert \int_{0}^{1}\mathbf{w}%
\cdot\left[  \nabla_{2}K^{\lambda}\left(  x,c  +t\left(
y-c\right)  \right)  -\nabla_{2}K^{\lambda}\left(  x,c\right)  \right]
dt\right\vert \leq C\ell\left(  I\right)  CC_{\operatorname*{CZ}}\ell\left(  I\right)
\frac{\delta^{n-\lambda+2}}{\ell\left(  I\right)  ^{n-\lambda+2}} =CC_{\operatorname*{CZ}}\frac{\delta^{n-\lambda+2}}{\ell\left(  I\right)
^{n-\lambda}}\,.
\end{align*}
On the other hand using (\ref{eq:gradient_elliptic}), we estimate
$\left\vert iii\right\vert $ from below by%
\begin{align*}
&  \left\vert y-c\right\vert \int_{0}^{1}\left[  \mathbf{w}\cdot\nabla
_{2}K^{\lambda}\left(  x,c\right)  \right]  dt=\left\vert y-c\right\vert
\left[  \mathbf{w}\cdot\nabla_{2}K^{\lambda}\left(  x,c\right)  \right] =\left\vert y-c\right\vert \left. \frac{d}{dt}K^{\lambda}\left(  x,x+t\mathbf{w}%
	\right)  \right \vert_{t=-\left\vert x-c\right\vert } \\ &\geq C_{\operatorname{grad}}%
\ell\left(  I\right)  \frac{1}{\left\vert x-c\right\vert ^{n-\lambda+1}}  \approx C_{\operatorname{grad}}\ell\left(  I\right)  \left(  \frac{\delta
}{\ell\left(  I\right)  }\right)  ^{n-\lambda+1}=C_{\operatorname{grad}}%
\delta^{n-\lambda+1}\left(  \frac{1}{\ell\left(  I\right)  }\right)
^{n-\lambda},
\end{align*}
since
\[
\left(  c-x\right)  \in S\left(  \mathbf{v},\delta\right)
,\ \ \ \ \ \text{for }0\leq t\leq1.
\]
Altogether then%
\begin{equation}
\left\vert i\right\vert +\left\vert ii\right\vert \leq\frac{1}{2}\left\vert
iii\right\vert \label{eq:negligible}%
\end{equation}
since%
\[
CC_{\operatorname*{CZ}}\frac{\delta^{n-\lambda+2}}{\ell\left(  I\right)
^{n-\lambda}}\leq\frac{1}{4}C_{\operatorname{grad}}\delta^{n-\lambda+1}%
\frac{1}{\ell\left(  I\right)  ^{n-\lambda}}%
\]
provided $\delta\leq\frac{1}{4}\frac{C_{\operatorname{grad}}}%
{CC_{\operatorname*{CZ}}}$. Thus the $\approx$ part of (\ref{sign}) holds if
we take $\delta=\min\left\{  \frac{1}{4}\frac{C_{\operatorname{grad}}%
}{CC_{\operatorname*{CZ}}},\delta_{0}\right\}  $. To see the equality of signs in (\ref{sign}), since $i$ and $ii$ are negligible by (\ref{eq:negligible}), then the
left side of (\ref{sign}) and $iii$ share the same sign. And (\ref{eq:assumption}) implies that $iii$ is positive, completing the proof of (\ref{sign}) when $y\in L$. The case of $y\in K$ is similar.

Thus the integrand in (\ref{T phi pointwise}) doesn't change sign, and 
$T^{\lambda}\varphi\left(  x\right)  $ is of one fixed sign for all $x\in J$.
Thus by \eqref{sign}, we have
\begin{align}
\left\vert T^{\lambda}\varphi\left(  x\right)  \right\vert  &  =\int
_{I}\left\vert K^{\lambda}\left(  x,y\right)  -K^{\lambda}\left(  x,c\right)
\right\vert \left\vert \frac{1}{\left\vert L\right\vert _{\sigma}}%
\mathbf{1}_{L}\left(  y\right)  -\frac{1}{\left\vert K\right\vert _{\sigma}%
}\mathbf{1}_{K}\left(  y\right)  \right\vert d\sigma\left(  y\right)
\label{eq:init_nondegen}\\
&  \geq\int_{I}\left\vert K^{\lambda}\left(  x,y\right)  -K^{\lambda}\left(
x,c\right)  \right\vert \left\vert \frac{1}{\left\vert L\right\vert _{\sigma}%
}\mathbf{1}_{L}\left(  y\right)  \right\vert d\sigma\left(  y\right)
\nonumber  \gtrsim\frac{1}{\ell\left(  I\right)  ^{n-\lambda}}\,,\nonumber
\end{align}
and hence
\begin{equation}\label{eq:nondege_-1}
\left\vert \left\langle T^{\lambda}\varphi,\mathbf{1}_{J}\right\rangle
_{\omega}\right\vert \gtrsim\frac{\left\vert J\right\vert _{\omega}%
}{\left\vert J\right\vert ^{1-\frac{\lambda}{n}}}\ .
\end{equation}

Now we note that $\frac{\varphi}{\left\Vert \varphi\right\Vert _{L^{2}\left(
\sigma\right)  }}$ is supported in $I$ with $\sigma$-mean zero, and is
constant on the dyadic $m$-grandchildren of $I$. The vector space of
such functions is the linear span of the Haar wavelets $\left\{  h_{M}%
^{\sigma,\tau}\right\}  _{M\in\mathfrak{C}_{\mathcal{D}}^{\left[  m-1\right]
}\left(  I\right)  ,\tau\in\Gamma_{M,n}^{\sigma}}$, where $\mathfrak{C}%
_{\mathcal{D}}^{\left[  m-1\right]  }\left(  I\right)  $ denotes all dyadic
grandchildren of $I$ down to level $m-1$. Since $\varphi$ belongs to this
vector space, we have the key identity,%
\begin{equation}
\frac{\varphi}{\left\Vert \varphi\right\Vert _{L^{2}\left(  \sigma\right)  }%
}=\sum_{M\in\mathfrak{C}_{\mathcal{D}}^{\left[  m-1\right]  }\left(  I\right)
,\gamma\in\Gamma_{M,n}^{\sigma}}\left\langle \frac{\varphi}{\left\Vert
\varphi\right\Vert _{L^{2}\left(  \sigma\right)  }},h_{M}^{\sigma,\gamma
}\right\rangle_{\sigma} h_{M}^{\sigma,\gamma}\ ,\label{key identity}%
\end{equation}
and so%
\begin{align}
\left\Vert T_{\sigma}^{\lambda}\frac{\varphi}{\left\Vert \varphi\right\Vert
_{L^{2}\left(  \sigma\right)  }}\right\Vert _{L^{2}\left(  \omega\right)  } &
\leq\sum_{M\in\mathfrak{C}_{\mathcal{D}}^{\left(  m-1\right)  }\left(
I\right)  ,\gamma\in\Gamma_{M,n}^{\sigma}}\left\vert \left\langle
\frac{\varphi}{\left\Vert \varphi\right\Vert _{L^{2}\left(  \sigma\right)  }%
},h_{M}^{\sigma,\gamma}\right\rangle_{\sigma} \right\vert \left\Vert T_{\sigma
}^{\lambda}h_{M}^{\sigma,\gamma}\right\Vert _{L^{2}\left(  \omega\right)
}\label{eq:phi_Haar_testing}\\
&  \leq\#\left\{  \left(  M,\gamma\right)  \in\mathfrak{C}_{\mathcal{D}%
}^{\left(  m-1\right)  }\left(  I\right)  \times\Gamma_{M,n}^{\sigma}\right\}
\mathfrak{H}_{T^{\lambda}}^{\mathcal{D},\operatorname*{glob}}\left(
\sigma,\omega\right)  \lesssim\mathfrak{H}_{T^{\lambda}}^{\mathcal{D}%
,\operatorname*{glob}}\left(  \sigma,\omega\right)  ,\nonumber
\end{align}
where we used Cauchy-Schwarz in the before-last inequality. Altogether we have%
\[
\frac{\left\vert J\right\vert _{\omega}}{\left\vert J\right\vert
^{1-\frac{\lambda}{n}}}\lesssim\left\vert \left\langle T_{\sigma}^{\lambda
}\varphi,\mathbf{1}_{J}\right\rangle _{\omega}\right\vert \lesssim\left\Vert
T_{\sigma}^{\lambda}\varphi\right\Vert _{L^{2}\left(  \omega\right)  }%
\sqrt{\left\vert J\right\vert _{\omega}}\lesssim\sqrt{\left\vert J\right\vert
_{\omega}}\left\Vert \varphi\right\Vert _{L^{2}\left(  \sigma\right)
}\mathfrak{H}_{T^{\lambda}}^{\mathcal{D},\operatorname*{glob}}\left(
\sigma,\omega\right)
\]
where $\left\Vert \varphi\right\Vert _{L^{2}\left(  \sigma\right)  }%
=\sqrt{\frac{1}{\left\vert L\right\vert _{\sigma}}+\frac{1}{\left\vert
K\right\vert _{\sigma}}}\approx\frac{1}{\sqrt{\left\vert I\right\vert
_{\sigma}}}$ by doubling, which gives
\[
\sqrt{\frac{\left\vert I\right\vert _{\sigma}\left\vert J\right\vert _{\omega
}}{\left\vert I\right\vert ^{1-\frac{\lambda}{n}}\left\vert J\right\vert
^{1-\frac{\lambda}{n}}}}\lesssim\mathfrak{H}_{T_{\sigma}^{\lambda}%
}^{\mathcal{D},\operatorname*{glob}}\left(  \sigma,\omega\right)  ,
\]
and since $\sigma,\omega$ are doubling, we obtain%
\begin{equation}\label{eq:A2_leq_Haar}
A_{2}^{\lambda}\left(  \sigma,\omega\right)  \lesssim\mathfrak{H}_{T^{\lambda
}}^{\mathcal{D},\operatorname*{glob}}\left(  \sigma,\omega\right)  .
\end{equation}

Now we turn to the triple testing characteristic $\mathfrak{T}_{T^{\lambda}%
}^{\mathcal{D},\operatorname*{trip}}\left(  \sigma,\omega\right)  $. Fix
$L\in\mathcal{D}$ and construct $K,I\in\mathcal{D}$ so that same previous conditions are satisfied for the triple of cubes $\left(  I,K,L\right)  $. Then 
\begin{align*}
\int_{3L}\left\vert T_{\sigma}^{\lambda}\varphi\right\vert ^{2}d\omega &
=\int_{3L}\left\vert \frac{1}{\left\vert L\right\vert _{\sigma}}T_{\sigma
}^{\lambda}\mathbf{1}_{L}-\frac{1}{\left\vert K\right\vert _{\sigma}}%
T_{\sigma}^{\lambda}\mathbf{1}_{K}\right\vert ^{2}d\omega\\
&  =\int_{3L}\left\vert \frac{1}{\left\vert L\right\vert _{\sigma}}T_{\sigma
}^{\lambda}\mathbf{1}_{L}\right\vert ^{2}d\omega+\int_{3L}\left\vert \frac
{1}{\left\vert K\right\vert _{\sigma}}T_{\sigma}^{\lambda}\mathbf{1}%
_{K}\right\vert ^{2}d\omega-2\frac{1}{\left\vert L\right\vert _{\sigma
}\left\vert K\right\vert _{\sigma}}\int_{3L}\left(  T_{\sigma}^{\lambda
}\mathbf{1}_{L}\right)  \left(  T_{\sigma}^{\lambda}\mathbf{1}_{K}\right)
d\omega,
\end{align*}
where by separation of $3L$ and $3K$, and the Calder\'{o}n-Zygmund size and
smoothness estimates, we have%
\begin{align*}
\left\vert \int_{3L}\left(  T_{\sigma}^{\lambda}\mathbf{1}_{L}\right)  \left(
T_{\sigma}^{\lambda}\mathbf{1}_{K}\right)  d\omega\right\vert  &  \lesssim
\int_{3L}\left\vert T_{\sigma}^{\lambda}\mathbf{1}_{L}\right\vert
\frac{\left\vert K\right\vert _{\sigma}}{\left\vert I\right\vert
^{1-\frac{\lambda}{n}}}d\omega\lesssim\frac{\left\vert K\right\vert _{\sigma
}\sqrt{\left\vert 3L\right\vert _{\omega}}}{\left\vert I\right\vert
^{1-\frac{\lambda}{n}}}\sqrt{\int_{3L}\left\vert T_{\sigma}^{\lambda
}\mathbf{1}_{L}\right\vert ^{2}d\omega}\\
&  \lesssim\sqrt{\int_{3L}\left\vert T_{\sigma}^{\lambda}\frac{\mathbf{1}_{L}%
}{\sqrt{\left\vert L\right\vert _{\sigma}}}\right\vert ^{2}d\omega}%
\frac{\left\vert K\right\vert _{\sigma}\sqrt{\left\vert L\right\vert _{\sigma
}\left\vert 3L\right\vert _{\omega}}}{\left\vert I\right\vert ^{1-\frac
{\lambda}{n}}}\lesssim\left\vert K\right\vert _{\sigma}A_{2}^{\lambda}\left(
\sigma,\omega\right)  \sqrt{\int_{3L}\left\vert T_{\sigma}^{\lambda}%
\frac{\mathbf{1}_{L}}{\sqrt{\left\vert L\right\vert _{\sigma}}}\right\vert
^{2}d\omega}.
\end{align*}
Thus, from the previous two displays and (\ref{eq:phi_Haar_testing}) we have%
\begin{align}
&  \int_{3L}\left\vert T_{\sigma}^{\lambda}\frac{\mathbf{1}_{L}}%
{\sqrt{\left\vert L\right\vert _{\sigma}}}\right\vert ^{2}d\omega=\left\vert
L\right\vert _{\sigma}\int_{3L}\left\vert T_{\sigma}^{\lambda}\frac
{\mathbf{1}_{L}}{\left\vert L\right\vert _{\sigma}}\right\vert ^{2}%
d\omega\lesssim\left\vert L\right\vert _{\sigma}\int_{3L}\left\vert T_{\sigma
}^{\lambda}\varphi\right\vert ^{2}d\omega+\frac{2}{\left\vert K\right\vert
_{\sigma}}\int_{3L}\left\vert T_{\sigma}^{\lambda}\mathbf{1}_{L}\right\vert
\left\vert T_{\sigma}^{\lambda}\mathbf{1}_{K}\right\vert d\omega
\label{eq:init_interval_to_Haar_bd}\\
&  \lesssim\mathfrak{H}_{T^{\lambda}}^{\mathcal{D},\operatorname*{glob}%
}\left(  \sigma,\omega\right)  ^{2}\left\vert L\right\vert _{\sigma}\int
_{3L}\left\vert \varphi\right\vert ^{2}d\sigma+A_{2}^{\lambda}\left(
\sigma,\omega\right)  \sqrt{\int_{3L}\left\vert T_{\sigma}^{\lambda}%
\frac{\mathbf{1}_{L}}{\sqrt{\left\vert L\right\vert _{\sigma}}}\right\vert
^{2}d\omega}\nonumber\\
&  \lesssim\mathfrak{H}_{T^{\lambda}}^{\mathcal{D},\operatorname*{glob}%
}\left(  \sigma,\omega\right)  ^{2}+A_{2}^{\lambda}\left(  \sigma
,\omega\right)  \sqrt{\int_{3L}\left\vert T_{\sigma}^{\lambda}\frac
{\mathbf{1}_{L}}{\sqrt{\left\vert L\right\vert _{\sigma}}}\right\vert
^{2}d\omega}\,.\nonumber
\end{align}
We conclude that%
\begin{align*}
\int_{3L}\left\vert T_{\sigma}^{\lambda}\frac{\mathbf{1}_{L}}{\sqrt{\left\vert
L\right\vert _{\sigma}}}\right\vert ^{2}d\omega & \leq C\mathfrak{H}%
_{T}^{\mathcal{D},\operatorname*{glob}}\left(  \sigma,\omega\right)
^{2}+CA_{2}^{\lambda}\left(  \sigma,\omega\right)  \sqrt{\int_{3L}\left\vert
T_{\sigma}^{\lambda}\frac{\mathbf{1}_{L}}{\sqrt{\left\vert L\right\vert
_{\sigma}}}\right\vert ^{2}d\omega}\\
& \leq C\mathfrak{H}_{T}^{\mathcal{D},\operatorname*{glob}}\left(
\sigma,\omega\right)  ^{2}+\frac{C}{\varepsilon}A_{2}^{\lambda}\left(
\sigma,\omega\right)  ^{2}+C\varepsilon\int_{3L}\left\vert T_{\sigma}%
^{\lambda}\frac{\mathbf{1}_{L}}{\sqrt{\left\vert L\right\vert _{\sigma}}%
}\right\vert ^{2}d\omega,
\end{align*}
and hence by absorbing the final term on the right into the left hand side
with $\varepsilon=\frac{1}{2C}$, and then taking the supremum over $L$, we
obtain
\[
\mathfrak{T}_{T^{\lambda}}^{\mathcal{D},\operatorname*{trip}}\left(
\sigma,\omega\right)  =\sup_{L\in\mathcal{D}}\sqrt{\int_{3L}\left\vert
T_{\sigma}^{\lambda}\frac{\mathbf{1}_{L}}{\sqrt{\left\vert L\right\vert
_{\sigma}}}\right\vert ^{2}d\omega}\lesssim\mathfrak{H}_{T}^{\mathcal{D}%
,\operatorname*{glob}}\left(  \sigma,\omega\right)  +A_{2}^{\lambda}\left(
\sigma,\omega\right)  .
\]
Combining this with \eqref{eq:A2_leq_Haar} gives%
\[
\mathfrak{T}_{T^{\lambda}}^{\mathcal{D},\operatorname*{trip}}\left(
\sigma,\omega\right)  \lesssim\mathfrak{H}_{T^{\lambda}}^{\mathcal{D}%
,\operatorname*{glob}}\left(  \sigma,\omega\right)  .
\]

\end{proof}

In \cite[Theorem 42]{AlLuSaUr} it was shown that for any single dyadic grid
$\mathcal{D}$ with doubling measures $\sigma$ and $\omega$ and an Stein elliptic operator $T^{\lambda}$, we can replace the testing characteristics $\mathfrak{T}%
_{T^{\lambda}}^{\operatorname{loc}}\left(  \sigma,\omega\right)  $ and
$\mathfrak{T}_{\left(  T^{\lambda}\right)  ^{\ast}}^{\operatorname{loc}%
}\left(  \omega,\sigma\right)  $ with dyadic testing
characteristics $\mathfrak{T}_{T^{\lambda}}^{\mathcal{D},\operatorname{loc}%
}\left(  \sigma,\omega\right)  $ and $\mathfrak{T}_{\left(  T^{\lambda
}\right)  ^{\ast}}^{\mathcal{D},\operatorname{loc}}\left(  \omega
,\sigma\right)  $, but the proof was quite complicated. Instead of appealing
to this result, we prove a weaker theorem that suffices for our purposes,
and with a much simpler proof, namely that we can replace the testing
characteristics with the larger \emph{triple} dyadic testing characteristics
$\mathfrak{T}_{T^{\lambda}}^{\mathcal{D},\operatorname*{trip}}\left(
\sigma,\omega\right)  $ and $\mathfrak{T}_{\left(  T^{\lambda}\right)  ^{\ast
}}^{\mathcal{D},\operatorname*{trip}}\left(  \omega,\sigma\right)  $.

\begin{proposition}
\label{single}Fix a dyadic grid $\mathcal{D}$ and assume notation as in Theorem \ref{T1_L2_Riesz}. If
the classical $T1$ theorem holds, i.e., \begin{equation}\label{eq:assume_T1}
\mathfrak{N}_{T^{\lambda}}\left(  \sigma,\omega\right)  \lesssim   A_{2}^{\lambda}\left(  \sigma,\omega\right)  +\mathfrak{T}%
_{T^{\lambda}}^{\operatorname*{loc}}\left(  \sigma,\omega\right)
+\mathfrak{T}_{\left(  T^{\lambda}\right)  ^{\ast}}^{\operatorname*{loc}}\left(  \omega,\sigma\right)  \,   ,
\end{equation}
then
\[
\mathfrak{N}_{T^{\lambda}}\left(  \sigma,\omega\right)  \lesssim   A_{2}^{\lambda}\left(  \sigma,\omega\right)  +\mathfrak{T}%
_{T^{\lambda}}^{\mathcal{D},\operatorname*{trip}}\left(  \sigma,\omega\right)
+\mathfrak{T}_{\left(  T^{\lambda}\right)  ^{\ast}}^{\mathcal{D}%
,\operatorname*{trip}}\left(  \omega,\sigma\right)  \,   .
\]

\end{proposition}

\begin{proof}
Given a doubling measure $\mu$ on $\mathbb{R}^{n}$ and $\varepsilon>0$, a
standard construction shows that there are positive constants $\eta, c_1 \in [0,1]$ and  $C_{0}, C_1 \in [1, \infty)$, depending only on the dimension $n$, the doubling constant of $\mu
$ and $\epsilon$, such that for any cube $I\in\mathcal{P}^{n}$, there are cubes $\left\{
J_{k}\right\}  _{k=1}^{C_{0}}\subset\mathcal{D}$ satisfying
\begin{equation}\label{eq:conds_Jk_cover}
\bigcup_{k=1}^{C_{0}}J_{k}\subset\eta I \, , \quad \left\vert I\setminus
\bigcup_{k=1}^{C_{0}}J_{k}\right\vert _{\mu}<\varepsilon\left\vert
I\right\vert _{\mu}\ \quad \text{and} \quad c_1 \ell \left ( I \right ) \leq \ell \left ( J_k \right ) \leq C_1 \ell \left ( I \right ) \, .
\end{equation}

In the computations that follow, we will track the dependence of certain constants on $\epsilon$ by a subscript: this dependence will be important to track since we will do an absorption argument at the end.

Thus from Minkowski's inequality, we have%
\[
\left(  \int_{I}\left\vert T_{\sigma}^{\lambda}\mathbf{1}_{I}\right\vert
^{2}d\omega\right)  ^{\frac{1}{2}}\leq\left(  \int_{I}\left\vert T_{\sigma
}^{\lambda}\mathbf{1}_{I\setminus\bigcup_{k=1}^{C_{0}}J_{k}}\right\vert
^{2}d\omega\right)  ^{\frac{1}{2}}+\sum_{k=1}^{C_{0}}\left(  \int
_{I}\left\vert T_{\sigma}^{\lambda}\mathbf{1}_{J_{k}}\right\vert ^{2}%
d\omega\right)  ^{\frac{1}{2}},
\]
where the square of the first term on the right hand side satisfies%
\begin{align*}
\int_{I}\left\vert T_{\sigma}^{\lambda}\mathbf{1}_{I\setminus\bigcup
_{k=1}^{C_{0}}J_{k}}\right\vert ^{2}d\omega &  \leq\mathfrak{N}_{T^{\lambda}%
}\left(  \sigma,\omega\right)  ^{2}\int_{I}\left\vert \mathbf{1}%
_{I\setminus\bigcup_{k=1}^{C_{0}}J_{k}}\right\vert ^{2}d\sigma\\
&  \leq\mathfrak{N}_{T^{\lambda}}\left(  \sigma,\omega\right)  ^{2}\left\vert
I\setminus\bigcup_{k=1}^{C_{0}}J_{k}\right\vert _{\sigma}\leq\varepsilon
\mathfrak{N}_{T^{\lambda}}\left(  \sigma,\omega\right)  ^{2}\left\vert
I\right\vert _{\sigma}\ .
\end{align*}
Now we turn to estimating the squares $\int_{I}\left\vert T_{\sigma}^{\lambda
}\mathbf{1}_{J_{k}}\right\vert ^{2}d\omega$ of the remaining terms for $1\leq
k\leq C_{0}$. We have%
\begin{align*}
\int_{I}\left\vert T_{\sigma}^{\lambda}\mathbf{1}_{J_{k}}\right\vert
^{2}d\omega &  =\int_{3J_{k}}\left\vert T_{\sigma}^{\lambda}\mathbf{1}_{J_{k}%
}\right\vert ^{2}d\omega+\int_{I\setminus3J_{k}}\left\vert T_{\sigma}%
^{\lambda}\mathbf{1}_{J_{k}}\right\vert ^{2}d\omega\\
&  \leq\mathfrak{T}_{T^{\lambda}}^{\mathcal{D},\operatorname*{trip}}\left(
\sigma,\omega\right)  ^{2}\left\vert J_{k}\right\vert _{\sigma}+C_{\epsilon} A_{2}%
^{\lambda}\left(  \sigma,\omega\right)  ^{2}\left\vert J_{k}\right\vert
_{\sigma},
\end{align*}
where the constant $C_{\epsilon}$ arises from the conditions \eqref{eq:conds_Jk_cover}. Summing in $k$ gives%
\[
\sum_{k=1}^{C_{0}}\left(  \int_{I}\left\vert T_{\sigma}^{\lambda}%
\mathbf{1}_{J_{k}}\right\vert ^{2}d\omega\right)  ^{\frac{1}{2}}\leq
\sqrt{C_{0}}\sqrt{\sum_{k=1}^{C_{0}}\int_{I}\left\vert T_{\sigma}^{\lambda
}\mathbf{1}_{J_{k}}\right\vert ^{2}d\omega}\leq C_{\epsilon} \left(  \mathfrak{T}%
_{T^{\lambda}}^{\mathcal{D},\operatorname*{trip}}\left(  \sigma,\omega\right)
+A_{2}^{\lambda}\left(  \sigma,\omega\right)  \right)  \sqrt{\left\vert
I\right\vert _{\sigma}} \, ,
\]
where we recall that $C_0$ depends on $\epsilon$ by \eqref{eq:conds_Jk_cover}.
Altogether then we have%
\[
\left(  \int_{I}\left\vert T_{\sigma}^{\lambda}\mathbf{1}_{I}\right\vert
^{2}d\omega\right)  ^{\frac{1}{2}}\leq C \left[  \sqrt{\varepsilon
}\mathfrak{N}_{T^{\lambda}}\left(  \sigma,\omega\right)  +C_{\epsilon}\mathfrak{T}%
_{T^{\lambda}}^{\mathcal{D},\operatorname*{trip}}\left(  \sigma,\omega\right)
+C_{\epsilon} A_{2}^{\lambda}\left(  \sigma,\omega\right)  \right]  \sqrt{\left\vert
I\right\vert _{\sigma}},
\]
which shows that%
\[
\mathfrak{T}_{T^{\lambda}}\left(  \sigma,\omega\right)  =\sup_{I\in
\mathcal{P}^{n}}\left(  \frac{1}{\left\vert I\right\vert _{\sigma}}\int
_{I}\left\vert T_{\sigma}^{\lambda}\mathbf{1}_{I}\right\vert ^{2}%
d\omega\right)  ^{\frac{1}{2}}\leq C \sqrt{\varepsilon}\mathfrak{N}%
_{T^{\lambda}}\left(  \sigma,\omega\right)  +C_{\epsilon}\mathfrak{T}_{T^{\lambda}%
}^{\mathcal{D},\operatorname*{trip}}\left(  \sigma,\omega\right)
+C_{\epsilon} A_{2}^{\lambda}\left(  \sigma,\omega\right)  .
\]
Now we conclude from this and (\ref{eq:assume_T1}) that
\begin{align*}
\mathfrak{N}_{T^{\lambda}}\left(  \sigma,\omega\right)   &  \leq C_{\lambda
,n}\left(  A_{2}^{\lambda}\left(  \sigma,\omega\right)  +\mathfrak{T}%
_{T^{\lambda}}\left(  \sigma,\omega\right)  +\mathfrak{T}_{\left(  T^{\lambda
}\right)  ^{\ast}}\left(  \omega,\sigma\right)  \right) \\
&  \leq C_{\lambda,n}\left(  C \sqrt{\varepsilon}\mathfrak{N}%
_{T^{\lambda}}\left(  \sigma,\omega\right)  +C_{\epsilon} A_{2}^{\lambda}\left(
\sigma,\omega\right)  +C_{\epsilon}\mathfrak{T}_{T^{\lambda}}^{\mathcal{D}%
,\operatorname*{trip}}\left(  \sigma,\omega\right)  +C_{\epsilon}\mathfrak{T}_{\left(
T^{\lambda}\right)  ^{\ast}}^{\mathcal{D},\operatorname*{trip}}\left(
\omega,\sigma\right)  \right)  .
\end{align*}
Since $T^{\lambda}$ is an admissible truncation, we have the \`a-priori estimate $\mathfrak{N}%
_{T^{\lambda}}\left(  \sigma,\omega\right)  <\infty$, and then absorbing the
term $C C_{\lambda,n} \sqrt{\varepsilon}\mathfrak{N}_{T^{\lambda}}\left(
\sigma,\omega\right)  $ into the left hand side with $\varepsilon \equiv \left(
\frac{1}{2 C C_{\lambda,n}}\right)  ^{2}$ finishes the proof of
Proposition \ref{single}.
\end{proof}

\begin{proof}
[Proof of Theorem \ref{main''}] We trivially have 
	\[
		\mathfrak{H}
_{ \mathbf{T}^{\lambda}}^{\mathcal{D},\operatorname*{glob}}\left(  \sigma,\omega\right)
 + \mathfrak{H}_{ \mathbf{T}^{\lambda}}^{\mathcal{D},\operatorname*{glob}}\left(
\omega,\sigma\right)  \lesssim \mathfrak{N}_{ \mathbf{T}^{\lambda}}\left(  \sigma
,\omega\right) \, .
	\]
Our assumption and Proposition \ref{single} followed by Theorem \ref{Haar} altogether give the opposite inequality,%
\begin{align*}
\mathfrak{N}_{\mathbf{T}^{\lambda}}\left(  \sigma,\omega\right)   &  \lesssim
A_{2}^{\lambda}\left(  \sigma,\omega\right)  +\mathfrak{T}_{ \mathbf{T}^{\lambda}
}^{\mathcal{D},\operatorname*{trip}}\left(  \sigma,\omega\right)
+\mathfrak{T}_{ (\mathbf{T}^{\lambda})^*}^{\mathcal{D}%
,\operatorname*{trip}}\left(  \omega,\sigma\right) \lesssim\mathfrak{H}_{ \mathbf{T}^{\lambda}}^{\mathcal{D},\operatorname*{glob}%
}\left(  \sigma,\omega\right)  +\mathfrak{H}_{ (\mathbf{T}^{\lambda})^*}^{\mathcal{D}%
,\operatorname*{glob}}\left(  \omega,\sigma\right)  .
\end{align*}

\end{proof}

\section{Alpert wavelets}

\label{section:Alpert}

In this section we consider a much wider
class of bases than the class of weighted Haar wavelets, namely the
weighted \emph{Alpert} wavelets. There is however a price to pay, namely that
the family of truncations of Calder\'{o}n-Zygmund operators must be further
restricted.

We first recall the construction of weighted Alpert wavelets in
\cite{RaSaWi}\footnote{See \cite{AlSaUr} for the correction of a small
oversight in \cite{RaSaWi}.}. The Alpert wavelets generalize the Haar wavelets
by permitting more vanishing moments, at the price of including polynomials in
the restrictions of the wavelets to children of cubes. Let $\mu$ be a locally
finite positive Borel measure on $\mathbb{R}^{n}$, and fix $\kappa
\in\mathbb{N}$. For each cube $Q$, denote by $L_{Q;\kappa}^{2}\left(
\mu\right)  $ the finite dimensional subspace of $L^{2}\left(  \mu\right)  $
that consists of linear combinations of the indicators of\ the dyadic children
$\mathfrak{C}\left(  Q\right)  $ of $Q$ multiplied by polynomials of degree
less than $\kappa$, and such that the linear combinations have vanishing $\mu
$-moments on the cube $Q$ up to order $\kappa-1$:%
\[
L_{Q;\kappa}^{2}\left(  \mu\right)  \equiv\left\{  f=%
{\displaystyle\sum\limits_{Q^{\prime}\in\mathfrak{C}\left(  Q\right)  }}
\mathbf{1}_{Q^{\prime}}p_{Q^{\prime};\kappa}\left(  x\right)  :\int
_{Q}f\left(  x\right)  x^{\beta}d\mu\left(  x\right)  =0,\ \ \ \text{for
}0\leq\left\vert \beta\right\vert <\kappa\right\}  ,
\]
where $p_{Q^{\prime};\kappa}\left(  x\right)  =\sum_{\beta\in\mathbb{Z}%
_{+}^{n}:\left\vert \beta\right\vert \leq\kappa-1\ }a_{Q^{\prime};\beta
}x^{\beta}$ is a polynomial in $\mathbb{R}^{n}$ of degree less than $\kappa$.
Here $x^{\beta}=x_{1}^{\beta_{1}}x_{2}^{\beta_{2}}...x_{n}^{\beta_{n}}$. Let
$d_{Q;\kappa}\equiv\dim L_{Q;\kappa}^{2}\left(  \mu\right)  $ be the dimension
of the finite dimensional linear space $L_{Q;\kappa}^{2}\left(  \mu\right)  $.

For $Q$, let
$\bigtriangleup_{Q;\kappa}^{\mu}$ denote orthogonal projection onto the finite
dimensional subspace $L_{Q;\kappa}^{2}\left(  \mu\right)  $, and let
$\mathbb{E}_{Q;\kappa}^{\mu}$ denote orthogonal projection onto the finite
dimensional subspace%
\[
\mathcal{P}_{Q;\kappa}^{n}\left(  \mu\right)  \equiv
\mathrm{\operatorname*{Span}}\{\mathbf{1}_{Q}x^{\beta}:0\leq\left\vert
\beta\right\vert <\kappa\}.
\]

For a doubling measure $\mu$, it is proved in \cite{RaSaWi} that for any dyadic grid $\mathcal{D}$, we have the
orthonormal decompositions%
\begin{equation}
f=\sum_{Q\in\mathcal{D}}\bigtriangleup_{Q;\kappa}^{\mu}f,\ \ \ \ \ f\in
L_{\mathbb{R}^{n}}^{2}\left(  \mu\right)  ,\ \ \ \ \ \text{where }\left\langle
\bigtriangleup_{P;\kappa}^{\mu}f,\bigtriangleup_{Q;\kappa}^{\mu}f\right\rangle
=0\text{ for }P\neq Q, \label{Alpert expan}%
\end{equation}
where convergence holds both in $L_{\mathbb{R}^{n}}^{2}\left(  \mu\right)  $
norm and pointwise $\mu$-almost everywhere, the telescoping identities%
\begin{equation}
\mathbf{1}_{Q}\sum_{I:\ Q\subsetneqq I\subset P}\bigtriangleup_{I;\kappa}%
^{\mu}=\mathbb{E}_{Q;\kappa}^{\mu}-\mathbf{1}_{Q}\mathbb{E}_{P;\kappa}^{\mu
}\ \text{ \ for }P,Q\in\mathcal{D}\text{ with }Q\subsetneqq P,
\label{telescoping}%
\end{equation}
and the moment vanishing conditions%
\begin{equation}
\int_{\mathbb{R}^{n}}\bigtriangleup_{Q;\kappa}^{\mu}f\left(  x\right)
\ x^{\beta}d\mu\left(  x\right)  =0,\ \ \ \text{for }Q\in\mathcal{D},\text{
}\beta\in\mathbb{Z}_{+}^{n},\ 0\leq\left\vert \beta\right\vert <\kappa\ .
\label{mom con}%
\end{equation}

We have the following bound for the Alpert projections $\mathbb{E}_{I;\kappa
}^{\mu}$ (\cite[see (4.7) on page 14]{Saw6}):
\begin{equation}
\left\Vert \mathbb{E}_{I;\kappa}^{\mu}f\right\Vert _{L_{I}^{\infty}\left(
\mu\right)  }\lesssim E_{I}^{\mu}\left\vert f\right\vert \leq\sqrt{\frac
{1}{\left\vert I\right\vert _{\mu}}\int_{I}\left\vert f\right\vert ^{2}d\mu
},\ \ \ \ \ \text{for all }f\in L_{\operatorname*{loc}}^{2}\left(  \mu\right)
. \label{analogue}%
\end{equation}
In terms of Alpert coefficient vectors $\widehat{f}\left(  I\right)
\equiv\left\{  \left\langle f,h_{I;\kappa}^{\mu,\gamma}\right\rangle \right\}
_{\gamma\in\Gamma_{I,n,\kappa}^{\mu}}$ for a choice of an orthonormal basis
$\left\{  h_{I;\kappa}^{\mu,\gamma}\right\}  _{\gamma\in\Gamma_{I,n,\kappa
}^{\mu}}$ of $L_{I;\kappa}^{2}\left(  \mu\right)  $, where $\Gamma
_{I,n,\kappa}$ is a convenient finite index set of size $d_{Q;\kappa}$, we
thus have%
\begin{equation}
\left\vert \widehat{f}\left(  I\right)  \right\vert =\left\Vert \bigtriangleup
_{I;\kappa}^{\sigma}f\right\Vert _{L^{2}\left(  \sigma\right)  }\leq\left\Vert
\bigtriangleup_{I;\kappa}^{\sigma}f\right\Vert _{L^{\infty}\left(
\sigma\right)  }\sqrt{\left\vert I\right\vert _{\sigma}}\leq C\left\Vert
\bigtriangleup_{I;\kappa}^{\sigma}f\right\Vert _{L^{2}\left(  \sigma\right)
}=C\left\vert \widehat{f}\left(  I\right)  \right\vert . \label{analogue'}%
\end{equation}
For notational convenience, we sometimes denote a choice $\left\{
h_{I;\kappa}^{\mu,\gamma}\right\}  _{\gamma\in\Gamma_{I,n,\kappa}^{\mu}}$ of
$\kappa$-Alpert wavelets by $\mathbf{a}_{\kappa}^{\mu}\equiv\left\{
h_{I;\kappa}^{\mu,\gamma}\right\}  _{\gamma\in\Gamma_{I,n,\kappa}^{\mu}}$. As
for the Haar wavelets, whenever we write $h_{I; \kappa}^{\mu}$, we will mean
an $L^{2} \left( \mu\right) $ normalized function in $L^{2} _{I; \kappa}
\left(  \mu\right) $. So as before, if we do not need to be precise about
which specific Alpert wavelet we're referring to, we'll simply write $h_{I;
\kappa}^{\mu}$ for an individual Alpert wavelet, or $\left\{  h_{I; \kappa
}^{\mu} \right\} $ for the collection of $\mu$-weighted Alpert wavelets
on an interval $I$.

Let $\kappa\in\mathbb{N}$, let $\mathcal{D}$ be a dyadic grid
and define the global and local dyadic  $\kappa$\emph{-Alpert} testing characteristics
for $T^{\lambda}$ by%
\begin{align*}
\mathfrak{A}_{T^{\lambda};\kappa}^{\mathcal{D}, \operatorname*{glob}}\left(  \sigma
,\omega\right)   &  \equiv\sup_{I\in\mathcal{D}} \sup\limits_{\left\{
h_{I; \kappa}^{\sigma} \right\}  }\left\Vert T_{\sigma}^{\lambda}h_{I;\kappa
}^{\sigma,\gamma}\right\Vert _{L^{2}\left(  \omega\right)  }\text{ and
}\mathfrak{A}_{T^{\lambda};\kappa}^{\mathcal{D}, \operatorname*{glob}}\left(  \omega
,\sigma\right)  \equiv\sup_{I\in\mathcal{D} } \sup\limits_{\left\{  h_{I;
\kappa}^{\omega} \right\}  } \left\Vert T_{\omega}^{\lambda,\ast}h_{I;\kappa
}^{\omega}\right\Vert _{L^{2}\left(  \sigma\right)  },\\
\mathfrak{A}_{T^{\lambda};\kappa}^{\mathcal{D},\operatorname{loc}}\left(  \sigma
,\omega\right)   &  \equiv\sup_{I\in\mathcal{D}} \sup\limits_{\left\{
h_{I; \kappa}^{\sigma} \right\}  } \left\Vert \mathbf{1}_{I}T_{\sigma
}^{\lambda}h_{I;\kappa}^{\sigma}\right\Vert _{L^{2}\left(  \omega\right)
}\text{ and }\mathfrak{A}_{T^{\lambda};\kappa}^{\mathcal{D}, \operatorname{loc}}\left(
\omega,\sigma\right)  \equiv\sup_{I\in\mathcal{D}} \sup\limits_{\left\{
h_{I; \kappa}^{\omega} \right\}  } \left\Vert \mathbf{1}_{I}T_{\omega
}^{\lambda}h_{I;\kappa}^{\omega}\right\Vert _{L^{2}\left(  \sigma\right)  } \, .
\end{align*}
Note that we have replaced $\mathfrak{H}$ = fraktur $H$ by
$\mathfrak{A}$ = fraktur $A$ and added the subscript $;\kappa$ when passing
from Haar to Alpert characteristics; indeed, when $\kappa=1$, we recover the familiar Haar testing characteristics. See \cite{RaSaWi,SaWi} for additional details on weighted Alpert wavelets.

Here we extend Theorem \ref{main'} to Alpert
orthonormal bases. Because of the additional parameter $\kappa$, then the constants implicit in the symbols $\lesssim, \approx$ and $\gtrsim$ may depend on $\kappa$.
\begin{theorem}
\label{main' Alpert}Let $\lambda \in [0, n) \setminus \{1\}$. If $\kappa\in\mathbb{N}$, $\sigma$ and $\omega$ are
doubling measures on $\mathbb{R}$ and $\mathcal{D}$ is a dyadic grid on
$\mathbb{R}$, then%
\[
\mathfrak{N}_{\mathbf{R}^{\lambda,n}}\left(  \sigma,\omega\right)  \approx
\mathfrak{A}_{\mathbf{R}^{\lambda,n};\kappa}^{\mathcal{D},\operatorname*{glob}}\left(
\sigma,\omega\right)  +\mathfrak{A}_{\mathbf{R}^{\lambda,n};\kappa}^{\mathcal{D}%
,\operatorname*{glob}}\left(  \omega,\sigma\right)  \ .
\]

\end{theorem}
As before, using the already existing $T1$ theorems for the vector Riesz transforms, these theorems boil down to proving the following theorem.

\begin{theorem}
 \label{thm:T1_to_Ta_all}  
 Let $\kappa\in\mathbb{N}$, $1<p<\infty$, $\lambda \in [0, n)$, 
$\sigma$ and $\omega$ be doubling measures on $\mathbb{R}$, $\mathcal{D}$ be
a dyadic grid on $\mathbb{R}$ and $T^{\lambda}$ be a $\lambda$-fractional Calder\' on-Zygmund operator. If
\begin{equation}
    \label{eq:Lp_testing_quad_2}
\mathfrak{N}_{ T^{\lambda}}\left(  \sigma,\omega\right)  \approx
\mathfrak{T}_{T^{\lambda}}^{\operatorname*{loc}}\left(
\sigma,\omega\right)  +\mathfrak{T}_{T^{\lambda}}%
^{\operatorname*{loc}}\left(  \omega,\sigma\right)
+A_{2}^{\lambda}\left(  \sigma,\omega\right)  \, ,
\end{equation}
then
\[
\mathfrak{N}_{T^{\lambda}}\left(  \sigma,\omega\right)  \approx
\mathfrak{A}_{T^{\lambda};\kappa}^{\mathcal{D},\operatorname*{glob}}\left(
\sigma,\omega\right)  +\mathfrak{A}_{T^{\lambda};\kappa}^{\mathcal{D}%
,\operatorname*{glob}}\left(  \omega,\sigma\right)  \ .
\]
\end{theorem}

The proof of Theorem \ref{thm:T1_to_Ta_all} is virtually identical to the proof of Theorem \ref{main''}. However there are two important differences. Recall that in the proof of
Theorem \ref{main''}, given dyadic cubes $I$ and $J$ of equal side length $\ell\left(  I\right)  $
with $\operatorname*{dist}\left(  I,J\right)  \approx\frac{1}{\delta}%
\ell\left(  I\right)  $, and centers $c_{I}$ and $c_{J}$, such that $J\subset
c_{I}+S\left(  \mathbf{v},\delta\right)  $, we showed there exist two
$m$-grandchildren $K,L\in\mathfrak{C}_{\mathcal{D}}^{\left(  m\right)
}\left(  I\right)  $, where $m \lesssim 1$, such that $\operatorname*{dist}\left(  K,L\right)
\approx
\ell\left(  I\right)  $ and $L\subset c_{K}+S\left(  \mathbf{v},\delta\right)
$. Then the function $\varphi$ defined in (\ref{def phi}) by%
\begin{equation}
\varphi\equiv\frac{1}{\left\vert L\right\vert _{\sigma}}\mathbf{1}_{L}%
-\frac{1}{\left\vert K\right\vert _{\sigma}}\mathbf{1}_{K}\ , \label{def phi'}%
\end{equation}
was shown to satisfy two critical properties.

\begin{enumerate}

\item The nondegeneracy inequality
(\ref{eq:init_nondegen}), i.e.,  $T_{\sigma}^{\lambda}\varphi\left(  x\right)$ is of one sign for $x \in J$ and 
\begin{align}\label{eq:nondegen_original}
\left\vert T_{\sigma}^{\lambda}\varphi\left(  x\right)  \right\vert
\gtrsim\frac{1}{\left|  I\right|  ^{1-\frac{\lambda}{n}} } \quad\text{
when } x \in J,
\end{align}
which followed because the integrand%
\begin{equation}
\left(  K^{\lambda}\left(  x,y\right)  -K^{\lambda}\left(  x,c\right)
\right)  \varphi\left(  y\right)  \label{T phi'}%
\end{equation}
in (\ref{T phi pointwise}) doesn't change sign if $K^{\lambda}$ is gradient elliptic, which here on we refer to as a ``positivity condition.''

\item That $\varphi$ is a linear combination of Haar wavelets associated to cubes for which $I$ is an ancestor at most $m$ generations above, i.e., the key identity (\ref{key identity}),
\begin{equation}
\frac{\varphi}{\left\Vert \varphi\right\Vert _{L^{2}\left(  \sigma\right)  }%
}=\sum_{M\in\mathfrak{C}_{\mathcal{D}}^{\left[  m-1\right]  }\left(  I\right)
,\gamma\in\Gamma_{M,n}^{\sigma}}\left\langle \frac{\varphi}{\left\Vert
\varphi\right\Vert _{L^{2}\left(  \sigma\right)  }},h_{M}^{\sigma,\gamma
}\right\rangle h_{M}^{\sigma,\gamma} \, ,
\end{equation}
	which followed because
\begin{enumerate}
\item $\frac{\varphi}{\left\Vert \varphi\right\Vert _{L^{2}\left(
\sigma\right)  }}$ is supported in $I$ with $\sigma$-mean zero, and is
constant on the $m$-grandchildren of $I$ for some $m \lesssim1$, and

\item the vector space of such functions is the linear span of the Haar
wavelets $\left\{  h_{M}^{\sigma,\tau}\right\}  _{M\in\mathfrak{C}%
_{\mathcal{D}}^{\left[  m-1\right]  }\left(  I\right)  }$.
\end{enumerate}
\end{enumerate}

In dealing with $\kappa$-Alpert wavelets instead of Haar wavelets, the
function $\varphi$ must be redefined so that

(\textbf{I}) the nondegeneracy condition (\ref{eq:init_nondegen}) continues to hold

(\textbf{II}) and $\varphi$ belongs to the linear span of the $\kappa$-Alpert
wavelets $\left\{  h_{K;\kappa}^{\sigma,\gamma}\right\}  _{K\in\mathfrak{C}%
_{\mathcal{D}}^{\left[  m-1\right]  }\left(  I \right)  ,\gamma
\in\Gamma\left(  K\right)  }$.

We focus on showing $\varphi$ is a dyadic step function on $I$ with bounded depth $m$ that satisfies the moment-vanishing conditions. Because $\varphi$ is a dyadic step function, higher order polynomials are not needed. To replace the positivity condition (\ref{T phi'}), we assume in addition that the kernel $K^{\lambda}\left(  x,y\right)  $ is $\kappa
$-elliptic, and use linear algebra to get the appropriate moment estimates.

\subsection{The dimension 1 argument for Theorem  \ref{thm:T1_to_Ta_all}}

We first work in the simpler case when the dimension $n=1$. Fix a configuration of intervals $I$ and $J$ in $\mathcal{D}$ as above, i.e., as in the proof of Theorem \ref{main'}, and by translating $\mathcal{D}$, assume without loss of
generality,
\[
I=\left[  0,\ell\left(  I\right)  \right]  \text{ lies to the left of }J.
\]
Let $m$ be a large positive integer that will be determined later; $m$ may only depend on $\kappa$, the doubling constants of $\sigma$ and $\omega$, and the Calder\'on-Zygmund data for $K$. To make the computations transparent, in what follows, the implicit constant $\approx, \lesssim$ and $\gtrsim$ will not be allowed to depend on $m$ \emph{until we fix  $m$.} For $h = 2^{-m} \ell \left (I \right )$, define 
\[
	K_j = \left [ \left (1+4 (j-1) \right )h, \left( 2+ 4(j-1) \right ) h \right ] \text{ for } 1 \leq j \leq \kappa +1 \, .
\]
Note $K_1 , \ldots, K_{\kappa+1}$ are dyadic subintervals of $I$, each of length $h$ and separated by $h$ as well.  Define 
\begin{equation}\label{eq:def_phi_1D}
	\varphi \equiv \frac{1}{ \left | I \right |_{\sigma}} \sum\limits_{i=1}^{\kappa+1} u_{i} \mathbf{1}_{K_i} \, ,
\end{equation}
where we will want to choose
\[
	\mathbf{u} \equiv \begin{bmatrix} u_1  \\ \vdots \\ u_{\kappa+1} \end{bmatrix} 
\] so that $\varphi$ satisfies the moment vanishing conditions up to order $\kappa-1$, i.e.,
\begin{equation}\label{eq:moment_vanishing_1d} 
	\int y^{\ell} \varphi (y) d \sigma (y)= 0  \text{ for } 0 \leq \ell \leq \kappa-1 \, . 
\end{equation}
The moment vanishing condition is equivalent to 
\[
	M \mathbf{u} \in \operatorname{Span} \left \{ \mathbf{e}_{\kappa+1} \right \} \, ,
\]
where we define the moment matrix
\begin{align*}
	M &  =\left[
\begin{array}
[c]{cccc}%
\int_{K_{1}}d\sigma\left(
y\right)   & \int_{K_{2}%
}d\sigma\left(  y\right)   & \cdots & \int_{K_{\kappa+1}}d\sigma\left(  y\right)  \\
\int_{K_{1}}yd\sigma\left(
y\right)   & \int_{K_{2}%
}yd\sigma\left(  y\right)   & \cdots & \int_{K_{\kappa+1}}yd\sigma\left(  y\right)  \\
\vdots & \vdots & \ddots & \vdots\\
\int_{K_{1}}y^{\kappa}%
d\sigma\left(  y\right)   & \int_{K_{2}}y^{\kappa}d\sigma\left(  y\right)   & \cdots & \int_{K_{\kappa+1}}%
y^{\kappa}d\sigma\left(  y\right)
\end{array}
\right] \, . 
\end{align*}
So let $\mathbf{u}$ be a unit vector for which 
\begin{align}\label{eq:Alpert_last_comp}
	M \mathbf{u} = \gamma \mathbf{e}_{\kappa+1} \text{ where } \gamma \in \mathbb{R} \, ,
\end{align}
whose existence is justified as follows: \eqref{eq:Alpert_last_comp}  is equivalent to 
\[
	\widetilde{M} \mathbf{u} = 0 \, ,
\]
where $\widetilde{M}$ is obtained by replacing the last row of $M$ by a row of zeros. Since $\widetilde{M}$ maps $\mathbb{R}^{\kappa+1}$ to a $\kappa$-dimensional space, it has nontrivial kernel which must contain a unit vector $\mathbf{u}$. Then with this choice of $\mathbf{u}$, w have $\varphi$ satisfies the moment vanishing conditions up to order $\kappa-1$, and $\varphi$ has $\kappa$-th moment equal to $\frac{\gamma}{\left | I \right |_{\sigma}}$, i.e.,
\begin{align}\label{eq:phi_nondegen_moment}
	\left | \int y^{\kappa} \varphi(y) d \sigma (y) \right | = \frac{\left | \gamma \right |}{\left | I \right |_{\sigma}} \, . 
\end{align}
Let us estimate $\left | \gamma \right |$ from below: by \eqref{eq:Alpert_last_comp} and Cramer's rule we have
\[
	 u_i = \frac{\det M^i}{\det M} \gamma  \, ,
\]
where $M^i$ is the matrix $M$ with $i$th column replaced by the column vector $\mathbf{e}_{\kappa+1}$. Thus for all $i$ we have
\[
	\left | \gamma \right | = |u_i| \frac{ | \det M |}{ | \det M^i |  } \, .
\]

The moment matrix $M$ has determinant%
\begin{align*}
	&  \det M =\sum_{\pi \in\operatorname*{Aut}\left(  \left\{
0,1,....,\kappa\right\}  \right)  }\left(  \operatorname{sgn}\pi\right)
\left(  \int_{K_{1}}y_{1}^{\pi\left(  0\right)  }d\sigma\left(
y_{1}\right)  \right)  \left(  \int_{K_{2}}y_{2}^{\pi\left(  1\right)
}d\sigma\left(  y_{2}\right)  \right)  ...\left(  \int_{K_{\kappa+1}}%
y_{\kappa+1}^{\pi\left(  \kappa\right)  }d\sigma\left(  y_{\kappa
+1}\right)  \right)  \\
&  =  \int_{K_{1}}\int_{K_{2}}...\int_{K_{\kappa+1}}\left\{
\sum_{\pi\in\operatorname*{Aut}\left(  \left\{  0,1....,\kappa\right\}
\right)  }\left(  \operatorname{sgn}\pi\right)  y_{1}^{\pi\left(
1\right)  }y_{2}^{\pi\left(  2\right)  }...y_{\kappa+1}^{\pi\left(
\kappa+1\right)  }\right\}  d\sigma\left(  y_{1}\right)  d\sigma\left(
y_{2}\right)  ...d\sigma\left(  y_{\kappa+1}\right)  \\
&  = \int_{K_{1}}\int_{K_{2}}...\int_{K_{\kappa+1}}\det\left[
\begin{array}
[c]{cccc}%
1 & 1 & \cdots & 1\\
y_{1} & y_{2} & \cdots & y_{\kappa+1}\\
\vdots & \vdots & \ddots & \vdots\\
y_{1}^{\kappa} & y_{2}^{\kappa} & \cdots & y_{\kappa+1}^{\kappa}%
\end{array}
\right]  d\sigma\left(  y_{1}\right)  d\sigma\left(  y_{2}\right)
...d\sigma\left(  y_{\kappa+1}\right)  ,
\end{align*}
which by Vandermonde's determinant theorem, equals 
\begin{align*}
&  =  \int_{K_{1}}\int_{K_{2}}...\int_{K_{\kappa+1}%
}\left\{  \prod_{1\leq i<j\leq\kappa+1}\left(  y_{j}-y_{i}\right)  \right\}
d\sigma\left(  y_{1}\right)  d\sigma\left(  y_{2}\right)  ...d\sigma\left(
y_{\kappa+1}\right)  \\
&  \approx   \int_{K_{1}}\int_{K_{2}}...\int_{K_{\kappa+1}%
}\left\{  \prod_{1\leq i<j\leq\kappa+1} h  \right\}
d\sigma\left(  y_{1}\right)  d\sigma\left(  y_{2}\right)  ...d\sigma\left(
	y_{\kappa+1}\right) = h^{\frac{\kappa\left(  \kappa+1\right)}{2}} \left \{ \prod\limits_{j=1}^{\kappa+1} \left | K_j \right |_{\sigma} \right \}   \, ,
\end{align*}
since each $h \leq y_{j}-y_{i}\leq 10 \kappa  h$ for $1\leq i<j\leq\kappa+1$.

Because $\det M^i$ and the determinant of the matrix $M$ with $i$th column and $(\kappa+1)$th row removed are equals in absolute value, then similarly
\[
	\left | \det M^i \right | \approx h^{\frac{\left ( \kappa -1\right ) \kappa }{2}} \left \{ \prod\limits_{j=1 \, , j \neq i}^{\kappa+1} \left | K_j \right |_{\sigma} \right \} \, .
\]

Thus, 
\begin{align}\label{eq:gamma_estimate}
	| \gamma | \approx \left | u_i \right | h^{\kappa} \left | K_i \right |_{\sigma} \text{ for all } 1 \leq i \leq \kappa+1 \, .  
\end{align}
Let $i$ be an index so that $\sum\limits_{j=1}^{\kappa+1} \left | u_j \right | \left | K_j \right |_{\sigma} \approx \left | u_i \right | \left | K_i \right |_{\sigma}$.
Now using Taylor's theorem and then the moment vanishing condition, for $x \in J$ we compute
\begin{align}
	\int K^{\lambda} (x,y) \varphi(y) d \sigma(y) &= \int \left \{ K^{\lambda} (x,0)+ \ldots + \partial_{2} ^{\kappa} K^{\lambda} (x,0)  \frac{y^{\kappa}}{\kappa!} + \partial_{2} ^{\kappa+1} K^{\lambda} (x,\theta y)  \frac{y^{\kappa+1}}{(\kappa+1)!} \right \}  \varphi(y) d \sigma(y) \label{eq:Taylor_Alpert_1d}\\
	&= \partial_{2} ^{\kappa} K^{\lambda} (x,0) \int     \frac{y^{\kappa}}{\kappa!} \varphi(y) d \sigma(y) + \int \partial_{2} ^{\kappa+1} K^{\lambda} (x,\theta y)  \frac{y^{\kappa+1}}{(\kappa+1)!} \varphi(y) d \sigma(y) \nonumber \, ,
\end{align}
where $\theta y$ denotes some point in $\left ( 0,y \right )$.
By \eqref{eq:phi_nondegen_moment}, \eqref{eq:gamma_estimate}, $\kappa$-ellipticity and the fact that $x \in J$ implies $|x-0| \approx \ell \left (I \right )$, we have  
\[
	\left |\partial_{2} ^{\kappa} K^{\lambda} (x,0) \int     \frac{y^{\kappa}}{\kappa!} \varphi(y) d \sigma(y) \right |\gtrsim \left ( \frac{h}{\ell \left(I \right )} \right)^{\kappa} \frac{ |u_i| \left | K_i \right |_{\sigma}}{ \ell \left ( I \right )^{1- \lambda} \left | I \right |_{\sigma} } \, .
\]And by the Calder\'on-Zygmund estimates, we have 
\begin{align*}
	&\left | \int \partial_{2} ^{\kappa+1} K^{\lambda} (x,\theta y)  \frac{y^{\kappa+1}}{(\kappa+1)!} \varphi(y) d \sigma(y) \right | \lesssim \left \| K^{\lambda} (x,\theta y)  y^{\kappa+1} \right \|_{L^{\infty} (\bigcup\limits_j K_j )} \left \| \varphi \right \|_{L^1 (\sigma)}  \\
 \lesssim &\left ( \frac{h}{\ell \left(I \right )} \right)^{\kappa+1} \frac{\sum\limits_{j=1}^{\kappa +1} \left | u_j \right | \left | K_j \right |_{\sigma}}{ \left | I \right |_{\sigma} \ell \left ( I \right )^{1-\lambda}} \approx \left ( \frac{h}{\ell \left(I \right )} \right)^{\kappa+1} \frac{\left | u_i \right | \left | K_i \right |_{\sigma}}{ \ell \left ( I \right )^{1- \lambda} \left | I \right |_{\sigma} } \, .
\end{align*}

Since $h = 2^{-m} \ell \left (I \right)$, then \emph{fixing} $m$ to be some sufficiently large constant depending only on $\kappa$, the doubling constants of $\sigma$ and $\omega$, and the Calder\'on-Zygmund data of $T$, yields the first term on the right of \eqref{eq:Taylor_Alpert_1d} dominates, $\int K^{\lambda} (x,y) \varphi(y) d \sigma(y)$ is of one sign, and 
\[
	\left | \int K^{\lambda} (x,y) \varphi(y) d \sigma(y) \right | \approx \left ( \frac{h}{\ell \left(I \right )} \right)^{\kappa} \frac{ |u_i| \left | K_i \right |_{\sigma}}{ \ell \left ( I \right )^{1- \lambda} \left | I \right |_{\sigma} } \approx \left ( \frac{h}{\ell \left(I \right )} \right)^{\kappa} \frac{1}{\ell \left ( I \right )^{1- \lambda}}  \sum\limits_j |u_j| \frac{\left | K_j \right |_{\sigma}}{ \left | I \right |_{\sigma}} \, .
\]
Since $m$ is fixed, any constant depending on $m$ may be absorbed into the symbols $\lesssim, \approx $ and $\gtrsim$. Since $h = 2^{-m} \ell \left (I \right)$, then 
\[
	\left | \int K ^{\lambda} (x,y) \varphi(y) d \sigma(y) \right | \approx \frac{1}{\ell \left (I \right )^{1- \lambda}}   \sum\limits_j |u_j|  \frac{\left | K_j \right |_{\sigma} }{ \left | I \right |_{\sigma} } \, .
\]
By the doubling of $\sigma$, we have $\left | K_j \right |_{\sigma} \approx \left | I \right |_{\sigma}$, which when combined with the fact that $\sum\limits_j |u_j| \approx 1$, yields the nondegeneracy condition \eqref{eq:nondegen_original}.

Finally, because $m \lesssim 1$, then $\varphi$ is a linear combination of indicators of dyadic descendants of $I$, all at most a bounded number of generations below $I$, and $\varphi$ satisfies the moment vanishing conditions. Hence $\varphi$ is a linear combination of a bounded number of $\kappa$-Alpert wavelets.

Following the proof of Theorem \ref{main'}, the proof of Theorem \ref{main' Alpert} is nearly complete in the case $n=1$: the rest of the proof of Theorem \ref{main'} remains nearly identical, except that we must find an appropriate substitute for  \eqref{eq:init_interval_to_Haar_bd}. To do so, fix some arbitrary dyadic interval and label it $K_i$, and then construct dyadic intervals $I$ and $\{K_j\}_{j \neq i}$ so that the same picture as before holds. More precisely, choose $I$ and $\{K_j\}_{j \neq i}$ so that all the intervals $K_j$ are all dyadic subintervals of $I$ of sidelength $2^{-m} \ell \left (I \right )$ and $\{3 K_j\}$ are all separated by $\approx \ell \left ( I \right )$; note that we do not require that $K_j$'s are all close to the left side of $I$. And then define $\varphi$ as in \eqref{eq:def_phi_1D}. Again, we choose the unit vector $\mathbf{u}$ so that $\varphi$ satisfies the moment-vanishing condition \eqref{eq:moment_vanishing_1d}. We can then apply Cramer's rule all over again, and using \eqref{eq:gamma_estimate}, we have for all $1 \leq j\leq \kappa +1$ that
\[
 |u_j| \left |I \right |_{\sigma} h^{\kappa}\approx  |u_j| \left |K_j \right |_{\sigma} h^{\kappa} \approx \left | \gamma \right | \approx |u_{\ell}| \left |K_{\ell} \right |_{\sigma} h^{\kappa} \approx \left |I \right |_{\sigma} h^{\kappa}  \, ,
\]
where we choose $\ell$ so that $\left | u_{\ell} \right | \approx 1$. Hence for all $j$ we then have
\[
|u_j| \approx 1 \, .
\] Then we can replace \eqref{eq:init_interval_to_Haar_bd} 
by the computation
\begin{align*}
 \int\limits_{3 K_1} \left | T_{\sigma} ^{\lambda } \frac{\mathbf{1}_{K_i}}{\sqrt{\left |K_i \right |_{\sigma}} }\right| ^2 d \omega \approx \int\limits_{3 K_i} \left | T_{\sigma} ^{\lambda } \frac{u_i \mathbf{1}_{K_i}}{\sqrt{\left |I \right |_{\sigma} }}\right| ^2 d \omega = \left | I \right |_{\sigma}  \int\limits_{3 K_i} \left | T_{\sigma} ^{\lambda } \frac{u_i \mathbf{1}_{K_i}}{\left |I \right |_{\sigma}} \right| ^2 d \omega \lesssim \left |I \right |_{\sigma} \int\limits_{3 K_i} \left | T_{\sigma} ^{\lambda} \varphi \right |^2 d \omega + \left |I \right |_{\sigma} \sum\limits_{j \neq i} \int\limits_{3 K_i} \left | T_{\sigma} ^{\lambda}  \frac{\mathbf{1}_{K_j}}{\left |I \right |_{\sigma} } \right |^2 \\
\lesssim \mathfrak{A}_{T^{\lambda}, \kappa} ^{\mathcal{D}, \operatorname{glob}} (\sigma, \omega) ^2 \left | I \right |_{\sigma} \int \left | \varphi \right |^2 d \sigma + A_2 ^{\lambda} (\sigma, \omega) \lesssim  \mathfrak{A}_{T^{\lambda}, \kappa} ^{\mathcal{D}, \operatorname{glob}} (\sigma, \omega) ^2 + A_2 ^{\lambda} (\sigma, \omega)   
 \, .
\end{align*}
And by
verifying that the triple Alpert testing characteristics are controlled using
the argument from the proof of Theorem \ref{Haar}, which does \emph{not} use
moment vanishing conditions, this then completes the proof of Theorem \ref{main' Alpert}. 

\subsection{The higher dimensional argument for Theorem  \ref{thm:T1_to_Ta_all}}

In extending the proof in dimension $1$ to dimension $n\geq 2$, we encounter two obstacles requiring key changes. One is that we have $\kappa$-ellipticity in only one of the $n$ directions, so we must configure our cubes so that the other directions are negligible when we do a Taylor expansion of $K^{\lambda}$. The other obstacle is the lack of a Vandermonde determinant identity. We get around both obstacles by dilating appropriately along the coordinate axes, since the determinant and the Taylor expansion both let us factor out the dilation factors, which helps us obtain quantitative estimates. 

Let $I$ and $J$ be cubes in the same configuration as before in $\mathcal{D}$, and assume without loss of generality that the origin is the furthest vertex of $I$ from $J$. We must construct a function $\varphi$ supported on finitely many dyadic subcubes of $I$ which has all moments vanishing up to order $\kappa-1$, i.e.,
\[
	\left \langle \varphi, y^{\alpha} \right \rangle_{\sigma} = \int \varphi (y) y^{\alpha} d \sigma \left ( y \right )  = 0 \text{ for all } \left | \alpha \right |\leq \kappa-1 \, .
\]
This is equivalent to $\varphi$ being in the $L^2 (\sigma)$-orthogonal complement of the linear span of the polynomials of degree at most $\kappa -1$. This latter set is invariant under a linear change of variables $x=Ay$, so $\varphi$ having vanishing moments up to order $\kappa-1$ is invariant under any linear change of variables.

Namely, we may change variables by a rotation or a reflection. By first rotating appropriately, assume without loss of generality that $K^{\lambda}$ is $\kappa$-elliptic with respect to the vector $\mathbf{v} = \mathbf{e}_1$. The grid $\mathcal{D}$ and the cubes $I$ and $J$ may no longer be axis-parallel, but this will not matter. And by then reflecting, assume without loss of generality that the center $c_I$ of $I$ lies in the horizontal half-cone
\[
\mathcal{K} \equiv \left \{ (x_1, \ldots, x_n ) \in \mathbb{R}^n ~:~ x_1 > \sqrt{\sum\limits_{j \geq 2} x_j ^2 } \text{ and } x_j \geq 0 \text{ for all } j\geq 2\right \} \, . 
\]
Note that if $x \in \mathcal{K}$, then for all $\lambda > 0$ sufficiently small, we have $\lambda x \in I$. 

We primarily do this change of variables to simplify notation with multi-indices. For instance, because of our rotation, note that $\kappa$-ellipticity for the kernel $K^{\lambda}$ means  
\[
	\left | \partial^{(\kappa ,0,\ldots, 0)} _{y} K^{\lambda} (x,y) \right | \approx \frac{1}{\left | x - y \right|^{n-\lambda + \kappa}} \, \text{ for all } y \in I, x \in J \, . 
\]
Indeed, this follows from an argument similar to the the reasoning below \eqref{perturb}, where one eventually plugs-in
\[
\mathbf{w} = \frac{x-y}{\left | x- y \right |} \, , \quad t = - \left | x- y \right | \, .
\]

 Let us establish some notation. Since we are working on $\mathbb{R}^n$, define the relevant sets of multi-indices 
\[
	\mathcal{E}^* \equiv \{ \beta \in \mathbb{N}^n~:~ |\beta| \leq \kappa-1 \} \text{ and } \mathcal{E} \equiv \mathcal{E} ^* \cup \{ (\kappa, 0,\ldots, 0)\} \, . \] Let $\prec$ be any ordering on the multi-indices such that 
\[
	 |\alpha| < |\beta| \text{ implies } \alpha \prec \beta \, .
 \] Note the smallest and largest elements in $\mathcal{E}$ with respect to $\prec$ are $\mathbf{0}$ and $\kappa \mathbf{e}_1$, respectively.  

Let $N \equiv \left | \mathcal{E} \right |$. We imbue $\mathbb{R}^{N}$ with the ordered orthonormal basis $\left ( \mathbf{e}_{\alpha} \right )_{\alpha \in \mathcal{E}}$, where the order is given by $\mathcal{E}$. Given $\alpha \in \mathcal{E}$, we also write  
\begin{equation}\label{eq:split_multi_indices}
	\alpha = (\alpha_1, \alpha') \in \mathbb{N} \times \mathbb{N}^{n-1}
\end{equation}
where $\alpha_1$ is the first coordinate of $\alpha$. We will need to consider $N \times N$ matrices  
\[
	M = \left ( M_{\alpha, \beta} \right )_{\alpha, \beta \in \mathcal{E}} \, :
\]
we adopt the convention that $\alpha$ and $\beta$ always parametrize the rows and columns of a matrix, respectively.

We will also consider collections of points $(x_{\beta})_{\beta \in \mathcal{E} }$ indexed by the multi-indices $\beta \in \mathcal{E}$, i.e., for each $\beta \in \mathcal{E}$ we have a point 
\[
	x_{\beta} = \left ( (x_{\beta})_1 ,\ldots , (x_{\beta})_n \right ) \in \mathbb{R}^n \, ,
\] for a total of $N = |\mathcal{E}|$ points in $\mathbb{R}^n$. While occasionally we will view the tuple $ (x_{\beta})_{\beta \in \mathcal{E}}$ as a point in $\left ( \mathbb{R}^n \right )^{N}$, we will more often think of $(x_{\beta})_{\beta \in \mathcal{E}}$ as $N$ points all in the \emph{same} Euclidean space $\mathbb{R}^n$.

Let $\delta \in \left (0, 1 \right )$ and $m \in \mathbb{N}$, and set $h = 2^{-m} \ell \left ( I \right)$: the large integer $m$ and the small constant $\delta$ will be chosen later, but will only depend on $\kappa$, the dimension $n$, the doubling constants of $\sigma$ and $\omega$, and the Calder\'on-Zygmund data for $T$. Again for clarity in the proof, until we have fixed $\delta$ and $m$, we will not allow the constants implicit in the symbols $\lesssim , \approx$ or $\gtrsim$ to depend on $\delta$ and $m$.

Define 
\[
	\varphi = \frac{1}{\left | I \right|_{\sigma}} \sum\limits_{\beta \in \mathcal{E}} u_{\beta} \mathbf{1}_{K_{\beta}} \, ,
\]
where the cubes $\{ K_{\beta}\}$ are dyadic subcubes of $I$ of sidelength $\delta^2 h$ which will be chosen later but should be thought of as follows: the cubes $K_{\beta}$ are distance $\lesssim \delta h$ away from the $\mathbf{e}_1$-axis, are distance $\approx h$ from the origin and each other. Also set $\mathbf{u} = \left (u_{\beta} \right )_{\beta \in \mathcal{E}} \in \mathbb{R}^N$.

\begin{figure}[ht]\label{fig:cubes_Alpert_slant}
  \fbox{\includegraphics[width=0.75\linewidth]{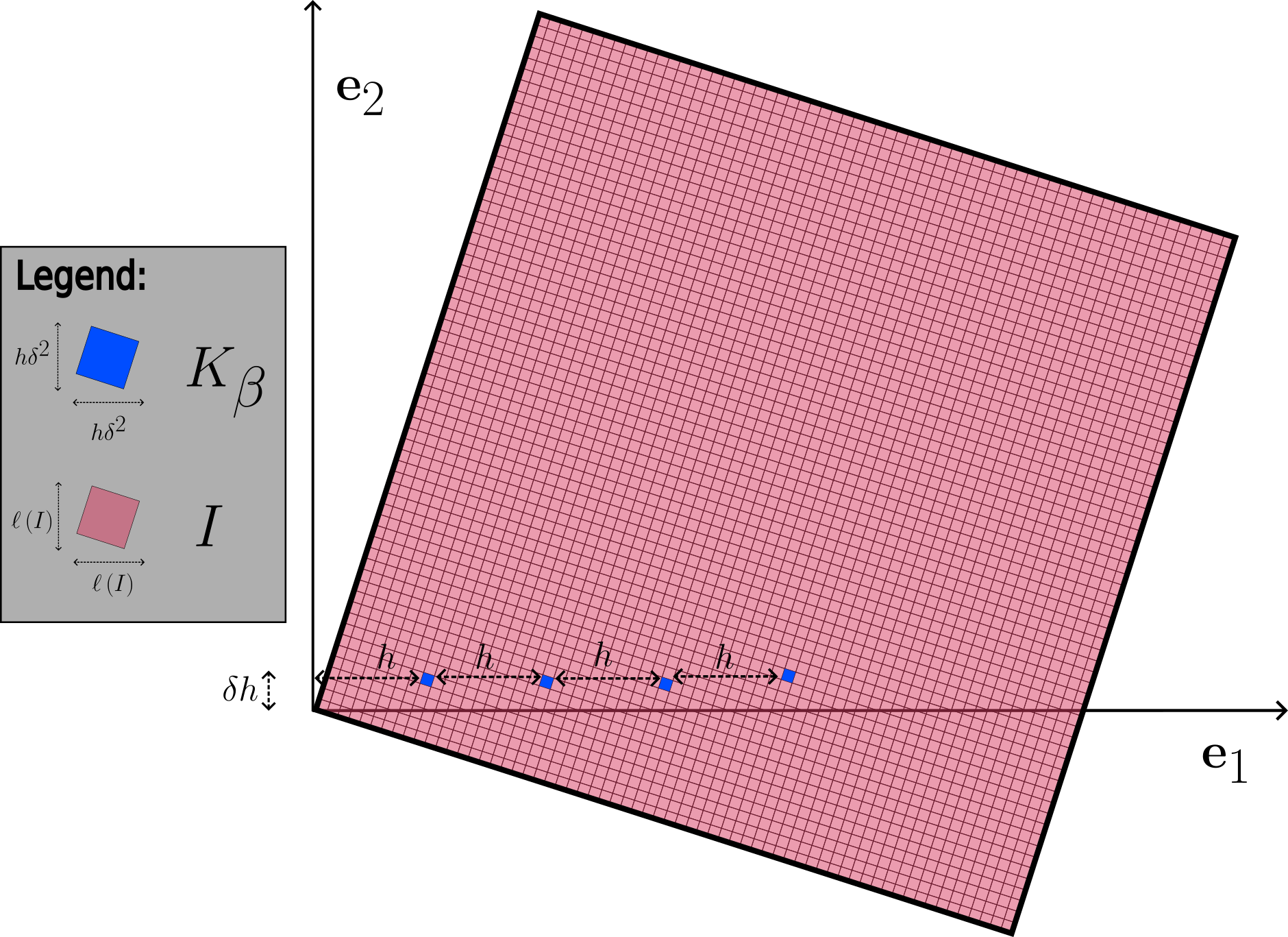}}
	\caption{The cubes $K_{\beta}$ within $I$. All distances depicted are $\approx$. }
\end{figure}

 Then the moment vanishing for $\varphi$ is equivalent to 
\begin{equation}\label{eq:moment_vanishing_general_matrix}
	M \mathbf{u} = \gamma \mathbf{e}_{(\kappa, 0, \ldots, 0)} \, , \quad \gamma \in \mathbb{R} \, ,
\end{equation}
where the moment matrix
\[
	M \equiv \begin{pmatrix} \int_{K_{\beta}} \left ( x_{\beta} \right )^{\alpha} d \sigma (x_{\beta}) \end{pmatrix}_{\alpha, \beta \in \mathcal{E}}  \, .
\]
Let $\mathbf{u}$ be a unit vector satisfying \eqref{eq:moment_vanishing_general_matrix}: such a unit vector exists because the matrix $\widetilde{M}$ obtained from $M$ by replacing its last row by a row of zeros must then have a nontrivial kernel because it maps an $N$-dimensional space into an $(N-1)$-dimensional space, and so we can take $\mathbf{u}$ to be a unit vector in the kernel. 

Since
\begin{equation}\label{eq:last_moment_genD}
\int y_1 ^{\kappa} \varphi (y) d \sigma (y) = \frac{\gamma}{\left | I \right |_{\sigma}} \, ,
\end{equation}
then to estimate the moment of $\varphi$ corresponding to the ellipticity of $K^{\lambda}$, it suffices to estimate $\gamma$. 
Cramer's rule applied to \eqref{eq:moment_vanishing_general_matrix} yields
\begin{equation}\label{eq:Cramer_high_dim}
	\left | u_{\beta} \right | = \left | \gamma \right | \frac{\left | \det M^{\beta} \right | }{\left | \det M \right |} \, ,
\end{equation}
where $M^{\beta}$ is simply the matrix $M$ with the $\beta$-th column vector replaced by the column vector $\mathbf{e}_{(\kappa, 0, \ldots, 0)}$.

 Because we lack a Vandermonde determinant formula in higher dimensions, we must estimate $\left | \det M \right |$ differently from before. First note that if we set 
\[
	\bar{x}_{\beta} \equiv \left ( \frac{(x_{\beta})_1}{h}, \frac{ (x_{\beta})_2}{\delta h}, \ldots,  \frac{ (x_{\beta})_n}{\delta h} \right )
\]
and let $H$ denote the diagonal matrix with $(\alpha, \alpha)$-th entry equal to $h^{|\alpha|} \delta^{|\alpha'|}$ (where $\alpha'$ is as in \eqref{eq:split_multi_indices}), 
then 
$M$ factors as 
\[
	M = H M' \, , \quad \text{ where } M' \equiv \begin{pmatrix} \int_{K_{\beta}} \bar{x}^{\alpha} _{\beta} d \sigma (x_{\beta}) \end{pmatrix}_{\alpha, \beta \in \mathcal{E}} \, . 
\]
Then 
\[
	\det H = h ^{\sum\limits_{\alpha \in \mathcal{E}} |\alpha|} \delta ^{\sum\limits_{\alpha \in \mathcal{E}} |\alpha'|} =  h ^{\kappa} \left ( h^{\sum\limits_{\alpha \in \mathcal{E}^*} |\alpha|} \delta ^{\sum\limits_{\alpha \in \mathcal{E}^*} |\alpha'|} \right ) \, .
\]
As for $M'$, we get
\[
	\det M' =  \det \begin{pmatrix} \int_{K_{\beta}}  \bar{x}^{\alpha} _{\beta} d \sigma (x_{\beta}) \end{pmatrix}_{\alpha, \beta \in \mathcal{E}} = \sum\limits_{\pi \in \operatorname{Aut} \left (\mathcal{E} \right )} \operatorname{sgn} (\pi) \prod\limits_{\beta \in \mathcal{E}} \int_{K_{\beta}}  \bar{x}^{\pi(\beta)} _{\beta} d \sigma (x_{\beta}) \, .
\]
Applying Fubini and noting that $\mathbf{x} \equiv \left ( x_{\beta} \right )_{\beta \in \mathcal{E}}$ belongs to the product space $\mathbf{K} \equiv \prod\limits_{\beta \in \mathcal{E}} K_{\beta}$ imbued with the product measure $\boldsymbol{\sigma} \equiv \otimes_{\beta \in \mathcal{E}} \sigma$, we see that 
\begin{equation}\label{eq:det_Fubini}
	\det M' = \int\limits_{\mathbf{K}} \sum\limits_{\pi \in \operatorname{Aut} \left (\mathcal{E} \right )} \operatorname{sgn} (\pi)   \bar{x}^{\pi(\beta)} _{\beta}  d  \boldsymbol{\sigma}  (\mathbf{x})  = \int\limits_{\mathbf{K} } \det M'' (\bar{\mathbf{x}})  d  \boldsymbol{\sigma}  (\mathbf{x}) \, ,
\end{equation}
where
\[
	\bar{\mathbf{x}} \equiv \left ( \bar{x}_{\beta} \right )_{\beta \in \mathcal{E}} \, ,
\]
and where we define the matrix function 
\[
	M'' \left (\left ( z_{\beta}\right )_{\beta \in \mathcal{E}} \right ) \equiv \begin{pmatrix} z_{\beta} ^{\alpha} \end{pmatrix}_{\alpha, \beta \in \mathcal{E}} \, .
\]
In what follows, we view $M''$ as a function of the point $\mathbf{z} \equiv (z_{\beta})_{\beta \in \mathcal{E}} \in \left (\mathbb{R}^{n} \right)^N$.
\begin{lemma}\label{lemma:det_nondegen}
	The function $\mathbf{z} \mapsto \det M'' (\mathbf{z})$ is nonzero on a dense subset of $\mathbb{R}^{Nn}$.	
\end{lemma}
\begin{proof}
	It suffices to show that on each open set $\mathbf{O}$ of the form
	\[
		\mathbf{O} \equiv \prod\limits_{\beta \in \mathcal{E}} \left ( O_{\beta} \right )_{\beta} \subset \mathbb{R}^{Nn} \, ,
	\]
	there exists a point $\mathbf{z} = \left ( z_{\beta} \right )_{\beta \in \mathcal{E}} \in \mathbf{O}$ such that $\det M'' \left (\mathbf{z} \right ) \neq 0$. 

	Assume to the contrary $\det M''$ vanishes everywhere on $\mathbf{O}$. Then all of its derivatives should vanish on $\mathbf{O}$. We will show that one of its derivatives is nonzero on $\mathbf{O}$, yielding a contradiction. Let $\alpha^* \equiv \kappa \mathbf{e}_1$ be the last index in $\mathcal{E}$. Then differentiating $\det M''$ with respect to $\partial^{\alpha^*} _{z_{\alpha^*}}$ and using the Laplace expansion of the determinant of $M''$ from the last column yields 
	\[
		\partial^{\alpha^*} _{z_{\alpha^*}} \det M'' \left ( \mathbf{z} \right )  =  c_{\alpha^*}   \det \left ( z_{\beta} ^{\alpha} \right )_{\alpha, \beta \in \mathcal{E}^*} \, , 
	\]
	where $c_{\alpha^*}$ is a nonzero constant independent of $\mathbf{z} = \left ( z_{\beta} \right )_{\beta \in \mathcal{E}}$.

	Then it suffices to show $\det \left ( z_{\beta} ^{\alpha} \right )_{\alpha, \beta \in \mathcal{E}^*}$ is not identically $0$ on the open set 
	\[
		\mathbf{O}^* \equiv \prod\limits_{\beta \in \mathcal{E}^*} O_{\beta} \subset \left ( \mathbb{R}^n \right)^{N-1}\, .
	\]
	If $\alpha^{**}$ is the last term in $\mathcal{E}^*$, then the Laplace expansion of $\det \left ( z_{\beta} ^{\alpha} \right )_{\alpha, \beta \in \mathcal{E}^*}$ from the last column again yields
	\[
		\partial_{z_{\alpha^{**}}} ^{\alpha^{**}} \det \left ( z_{\beta} ^{\alpha} \right )_{\alpha, \beta \in \mathcal{E}^*} = c_{\alpha ^{**}} \det \left ( z_{\beta} ^{\alpha} \right )_{\alpha, \beta \in \mathcal{E}^{**}} \ , 
	\]
	where $\mathcal{E}^{**}$ is missing the last two terms of $\mathcal{E}$. Iterating this process and recognizing the multi-index $\mathbf{0}$ as the first element in $\mathcal{E}$, we see we must show that
	\[
		\det \left ( z_{\beta} ^{\alpha} \right )_{\alpha, \beta \in \{ \mathbf{0} \}} 
	\]
	is not identically $0$ on $O_{\mathbf{0}}$. This determinant equals $1$ since $\left ( z_{\beta} ^{\alpha} \right )_{\alpha, \beta \in \{ \mathbf{0} \}}$ is a $1 \times 1$ matrix with constant entry $1$.	
	\end{proof}
Clearly, Lemma \ref{lemma:det_nondegen} also holds if we replace $M''\left ( \mathbf{z} \right ) $ by $M''_{\gamma} (\mathbf{z})$, where
\[
M'' _{\gamma} \left ( \mathbf{z} \right ) \equiv \left ( z_{\beta} ^{\alpha} \right )_{\alpha \in \mathcal{E} ^*, \beta \in \mathcal{E} \setminus \{ \gamma \} } \, . 
\]
\begin{corollary}\label{cor:choose_balls}
There exists balls $\{ B_{\beta} \}_{\beta \in \mathcal{E}}$ in the half-cone $\mathcal{K}$  with centers $\{z_{\beta}\}_{\beta \in \mathcal{E}}$  and identical radii $\epsilon (n, \kappa)$, such that whenever we have points $(y_{\beta})_{\beta \in \mathcal{E}}$ such that $y_{\beta} \in B_{\beta}$ for each $\beta$, then
\begin{equation}\label{eq:det_nondegen_cor}
\left | \det (y_{\beta}^{\alpha} )_{\alpha, \beta \in \mathcal{E}} \right| \approx 1 \, ,
\end{equation}
and the function in between absolute values is of one sign.

Furthermore, the balls are all contained in the $1$-neighborhood of the $\mathbf{e}_1$-axis, are all separated from each other by distance $\approx 1$, and are distance $\approx 1$ from the origin.
\end{corollary}
Note that the balls in the above corollary exist and are chosen independently of $I$ and $K_{\beta}$: their choice only depends on $n$ and $ \kappa$.
\begin{proof}
First pick a configuration of points $(z_{\beta})_{\beta \in \mathcal{E}}$ in the half-cone $\mathcal{K}$: the specific choice configuration does not matter, we only need that the points are all separated by distance $\approx 1$ and are distance $\approx 1$ from the origin. To be concrete, we take $(z_{\beta})_{\beta \in \mathcal{E}}$ to be the first $N$ points in the collection 
\[
	\left \{ \frac{1}{2} \mathbf{e}_2 + t \mathbf{e}_1 \in \mathbb{R}^n  ~:~ t\in \mathbb{Z}_{+} \right \} \, . 
\]
 By Lemma \ref{lemma:det_nondegen}, we may perturb the points $(z_{\beta})_{\beta \in \mathcal{E}}$ slightly, by a distance of at most $\frac{1}{8}$, so that the quantity
 \[
 \left | \det (z_{\beta}^{\alpha} )_{\alpha, \beta \in \mathcal{E}} \right | \neq 0 \, .
 \] Then define $c(\kappa, n)$ to be the value of the left-side: $c(\kappa, n)$ is some positive constant only depending on $\kappa$ and $n$.
By continuity, there exists some constant $\epsilon = \epsilon (\kappa, n) \in (0, \frac{1}{8} )$, only depending on $\kappa$ and $n$, such that whenever $y_{\beta}$ is a point in the ball 
\[
 B_{\beta} \equiv B_{\epsilon } ( z_{\beta}) \, ,
\]
we have
\begin{equation}\label{eq:det_nondegen}
\left | \det (y_{\beta}^{\alpha} )_{\alpha, \beta \in \mathcal{E}} \right| \approx 1 \, ,
\end{equation}
and the function in between absolute values is of one sign. Note the balls $B_{\beta}$ are all contained in the half-cone $\mathcal{K}$, are each separated by distance $\approx 1$, and are distance $\approx 1$ from the origin.
\end{proof}
\begin{remark}\label{rmk:simultaneous}
    Using the sentence just before Corollary \ref{cor:choose_balls}, one can easily modify the proof of Corollary \ref{cor:choose_balls} so that when $y_{\beta} \in B_{\beta}$ for all $\beta$, we simultaneously have \eqref{eq:det_nondegen_cor} and 
\[
\left | M_{\gamma} '' ( \left ( y_{\beta} ^{\alpha} \right )_{\alpha, \beta \in \mathcal{E} } ) \right | =  \left | \det \left ( y_{\beta} ^{\alpha} \right )_{\alpha \in \mathcal{E}^*, \beta \in \mathcal{E} \setminus \{\gamma \} } \right | \approx 1 \text{ for all } \gamma \in \mathcal{E} \, ,
\]
and each function in between the absolute values does not change sign.
\end{remark}

Define the function 
\[
	\hat{x} \equiv \frac{1}{h} x = \frac{1}{2^{-m} \ell \left ( I \right )} x  \, .
\]
Given $A \subset \mathbb{R}^n$, define 
\[
	\hat{A} \equiv \left \{ \hat{x} ~:~ x \in A \right \} \, . 
\]
Then $\hat{I}$ is the cube $I$ but re-scaled to have sidelength $2^{m}$. Because the center $c_I$ of $I$ belongs to $\mathcal{K}$, then so does the center $c_{\hat{I}}$ of $\hat{I}$. Choose $m$ large enough so that $c_{\hat{I}}$ extends past the balls $B_{\beta}$ along the $\mathbf{e}_1$-axis, i.e., 
\[
	\left( c_{\hat{I}} \right )_1 \geq \sup\limits_{\beta \in \mathcal{E}}\sup\limits_{y_{\beta} \in B_{\beta}} \left (y_{\beta} \right )_{1} \, . 
\]
Note then that each $B_{\beta} \subset \hat{I}$. We also note that by this re-scaling, the dyadic descendants of $\hat{I}$ of generation $m$ are of unit length. 

Given $\hat{x} \in \mathbb{R}^n$ and $\delta < \frac{1}{2}$, define
\[
	\bar{x} \equiv \left ( \hat{x}_1, \frac{\hat{x}_2}{\delta} , \ldots, \frac{\hat{x}_n}{\delta} \right ) \, , 
\]
which coincides with the definition
\[
	\bar{x} = \left ( \frac{x_1}{h}, \frac{x_2}{ \delta h}, \ldots, \frac{x_n}{ \delta h} \right ) \, . 
\]
And again given a set $\hat{A} \subset \mathbb{R}^n$, define
\[
	\bar{A} \equiv \left \{ \bar{x} ~:~ \hat{x} \in \hat{A} \right \} \, . 
\] 
\begin{figure}[ht]\label{fig:cubes_Alpert_slant_stretch}
  \fbox{\includegraphics[width=0.75\linewidth]{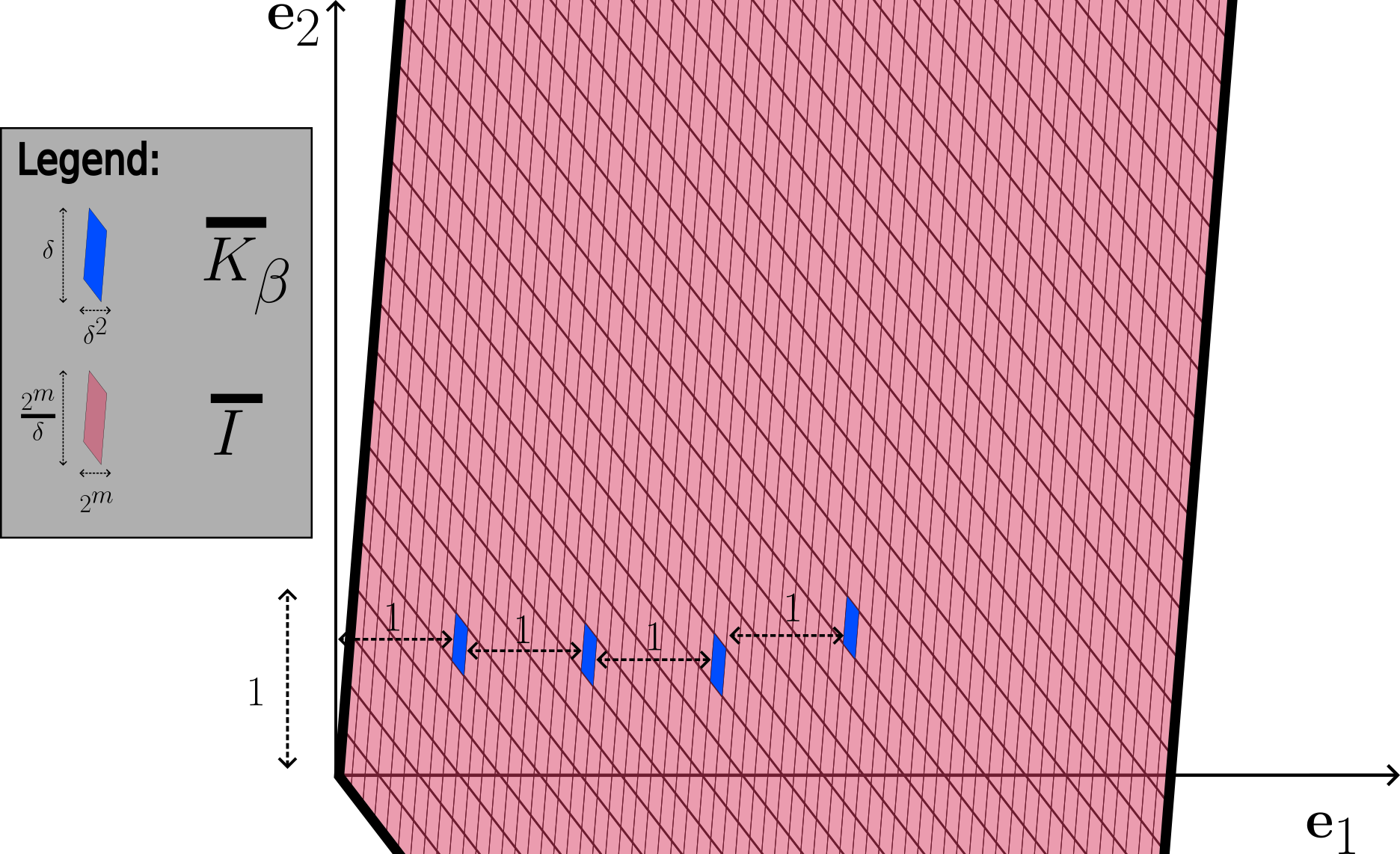}}
	\caption{The stretched cubes $\overline{K}_{\beta}$ within $\overline{I}$. All depicted distances are $\approx$.}
\end{figure}Because each ball $B_{\beta} \subset \hat{I} \subset \bar{I}$, then by Corollary \ref{cor:choose_balls} we have that
\begin{equation}\label{eq:pre-cover}
\left ( \bigcup\limits_{\beta} B_{\beta} \right ) \subset \bar{I} \cap \left \{z \in \mathbb{R}^n ~:~ 0 \leq \operatorname{dist} \left ( z, \mathbf{e}_1\text{-axis} \right ) \leq 1 \right \} \, . 
\end{equation}
The set $\{K\}$ of dyadic subcubes of $I$ of sidelength $h \delta^2$ cover $I$. Their images $\{\hat{K}\}$ then cover $\hat{I}$, and hence their images $\{\bar{K}\}$ cover $\bar{I}$. Since the right side of \eqref{eq:pre-cover} is a subset of $\bar{I}$, then it is covered by $\{\bar{K}\}$.

Let $\bar{K}_{\beta}$ denote the element of $\{\bar{K}\}$ containing the point $z_{\beta}$. Note that $\bar{K}_{\beta}$ is an object of dimensions $\approx \delta^2$ in the $\mathbf{e}_1$ direction, and $\approx \delta$ in the directions $\mathbf{e}_2, \ldots, \mathbf{e}_n$. For $\delta$ sufficiently small and only depending on the radius $\epsilon = \epsilon (n, \kappa)$ of the balls $B_{\beta}$, we in fact have that $\bar{K}_{\beta} \subset B_{\beta}$ for each $\beta$. By Corollary \ref{cor:choose_balls}, if $\bar{x}_{\beta} \in \bar{K}_{\beta}$, then 
\begin{equation}\label{eq:det_nondegen_bars}
	|\det M'' ( (\bar{x}_{\beta} )_{\beta\in \mathcal{E}} ) | \approx 1 \, ,
\end{equation}
and the determinant in between absolute values is of one sign.

Then $\bar{K}_{\beta}$ is the image of some dyadic subcube $K_{\beta}$ of $I$ of sidelength $h \delta^2$ under the composition of the $\hat{~}$ and  $\bar{~}$ maps. Then for all $x_{\beta} \in K_{\beta}$, we have $\bar{x}_{\beta} \in \bar{K}_{\beta}$ and so \eqref{eq:det_nondegen_bars} still holds, and the determinant is of one sign. Furthermore, each $K_{\beta}$ is of distance $\lesssim h \delta$ from the $\mathbf{e}_1$ axis and distance $\approx h$ from the origin.

Thus for $\mathbf{x} = (x^{\beta})_{\beta} \in \mathbf{K}$, by \eqref{eq:det_Fubini} we have
\[
	\left | \det M' \right | \approx \prod\limits_{\beta} \left | K_{\beta} \right |_{\sigma} \, ,
\]
and hence
\[
	\left | \det M \right | = \left | \det H \right | \left |  \det M' \right | \approx h ^{\kappa} \left ( h^{\sum\limits_{\alpha \in \mathcal{E}^*} |\alpha|} \delta ^{\sum\limits_{\alpha \in \mathcal{E}^*} |\alpha'|} \right ) \prod\limits_{\beta \in \mathcal{E}} \left | K_{\beta } \right |_{\sigma} \, . 
\]
Similarly, by the Laplace expansion of the determinant, we have $\left | \det M^{\beta} \right |$ equals the absolute value of the matrix obtained by deleting the $\beta$-th column and last row of $M$. From this last matrix, we can factor out a diagonal matrix $H$ whose $(\alpha, \alpha)$-th entry equals $h^{|\alpha|} \delta^{|\alpha'|}$, where $\alpha$ varies over $\mathcal{E}^*$ now, and applying the same logic as above, and in particular by applying Remark \ref{rmk:simultaneous}, we may also obtain
\[
	\left | \det M^{\beta} \right | \approx h ^{\sum\limits_{\alpha \in \mathcal{E}^*} |\alpha|} \delta ^{\sum\limits_{\alpha \in \mathcal{E}^*} |\alpha'|} \prod\limits_{\alpha \in \mathcal{E} ~:~ \alpha \neq \beta } \left | K_{\alpha} \right |_{\sigma} \, .
\]
Thus with our initial application of Cramer's rule \eqref{eq:Cramer_high_dim}, we get
\[
	\left | u_{\beta} \right | = \left | \gamma \right | \left | \frac{\det M^{\beta}}{\det M} \right | \approx \left | \gamma \right | \frac{1}{h^{\kappa} \left | K_{\beta}\right |_{\sigma}} \, ,
\]
or rather
\begin{equation}\label{eq:gamma_many_beta}
	\left | \gamma \right | \approx \left | u_{\beta} \right | \left | K_{\beta} \right |_{\sigma} h^{\kappa}  \text{ for all } \beta \in \mathcal{E} \, .
\end{equation}

Now let $\beta^* \in \mathcal{E}$ be a multi-index such that $\sum\limits_{\beta \in \mathcal{E} } \left | u_{\beta} \right | \left | K_{\beta} \right |_{\sigma} \approx \left | u_{\beta^*} \right | \left | K_{\beta^*} \right |_{\sigma}$. Using Taylor's theorem with $x \in J$, along with the moment-vanishing of $\varphi$ we then get $\int K^{\lambda} (x,y) \varphi (y) d\sigma(y)$ equals 
\begin{align*}
	\int \left \{ \partial^{\left ( \kappa, 0, \ldots, 0 \right )} _y K ^{\lambda} (x,0) \frac{y_1^{\kappa}}{\kappa!}  +  \sum\limits_{|\alpha| = \kappa ~:~ \alpha \neq \kappa \mathbf{e}_1 } c_{\alpha} \partial^{\alpha} _y  K^{\lambda} (x,0) y^{\alpha} +  \sum\limits_{|\alpha| = \kappa +1 } c_{\alpha} \partial^{\alpha} _y  K ^{\lambda} (x,\theta y) y^{\alpha}\right \} \varphi (y) d \sigma(y) \, . 
\end{align*}
Using \eqref{eq:gamma_many_beta}, \eqref{eq:last_moment_genD},  $\kappa$-ellipticity, and the fact that $\operatorname{dist}(\operatorname{supp} \varphi, 0) \lesssim h$,  we estimate the first term by 
\[
	\left | \partial^{\left ( \kappa, 0, \ldots, 0 \right )} _y K ^{\lambda} (x,0) \int  \frac{y_1^{\kappa}}{\kappa!} \varphi (y) d \sigma (y) \right |\approx \left | \partial^{(\kappa, 0, \ldots, 0)} _{y} K ^{\lambda} (x,0) \right | \left | \int   y_1^{\kappa} \varphi(y) \right | \approx  \left ( \frac{h}{\ell \left (I \right)} \right )^{\kappa} \frac{ \left | u_{\beta^*} \right | \left | K_{\beta^*} \right |_{\sigma}}{  \left |I \right |^{1 - \frac{\lambda}{n}} \left | I \right|_{\sigma}}   \, .
\]
 We estimate the second term using the Calder\'on-Zygmund estimates and that each cube $K_{\beta}$ is distance $\leq h \delta$ from the $\mathbf{e}_1$-axis: 
\[
	\left |  \sum\limits_{|\alpha| = \kappa ~:~ \alpha \neq \kappa \mathbf{e}_1 } c_{\alpha} \partial^{\alpha} _y  K ^{\lambda} (x,0) \int  y^{\alpha}\varphi (y)  d \sigma(y)     \right |	\lesssim \left ( \frac{h}{\ell \left (I \right)} \right )^{\kappa} \frac{\delta}{ \left |I \right |^{1 - \frac{\lambda}{n}} } \sum\limits_{\beta \in \mathcal{E}} \left | u_{\beta} \right |  \frac{\left | K_{\beta} \right |_{\sigma} }{  \left | I \right|_{\sigma} }   \approx \left ( \frac{h}{\ell \left (I \right)} \right )^{\kappa} \frac{\delta \left | u_{\beta^*} \right | \left | K_{\beta^*} \right |_{\sigma} }{ \left |I \right|^{1 - \frac{\lambda}{n}} \left | I \right |_{\sigma } }   \, .
\]
The third term can also be estimated by the Calder\'on-Zygmund estimates, 
\[
	\left | \int \sum\limits_{|\alpha| = \kappa +1 } c_{\alpha} \partial^{\alpha} _y  K ^{\lambda} (x,\theta y) y^{\alpha} \varphi (y) d \sigma(y) \right | \lesssim \left ( \frac{h}{\ell \left (I \right)} \right )^{\kappa+1}  \frac{1}{ \left |I \right |^{1 - \frac{\lambda}{n}} } \sum\limits_{\beta \in \mathcal{E}} \left | u_{\beta} \right |  \frac{\left | K_{\beta} \right |_{\sigma} }{  \left | I \right|_{\sigma} }    \approx \left ( \frac{h}{\ell \left (I \right)} \right )^{\kappa+1} \frac{ \left | u_{\beta^*} \right | \left | K_{\beta^*} \right |_{\sigma} }{ \left |I \right|^{1 - \frac{\lambda}{n}} \left | I \right |_{\sigma } }  \]

Fixing $\delta$ and $2^{-m} = \frac{h}{\ell \left (I \right)}$ to be sufficiently small constants, we see that the first term dominates, that $T_{\sigma} \varphi (x)$ is of one consistent sign for $x \in J$ and 
\[
	\left | T_{\sigma} ^{\lambda} \varphi (x) \right | = \left | \int K ^{\lambda } (x,y) \varphi (y) d\sigma(y) \right | \approx \left ( \frac{h}{\ell \left (I \right)} \right )^{\kappa} \frac{ \left | u_{\beta^*} \right | \left | K _{\beta^*} \right | _{\sigma} }{ \left |I \right|^{1 - \frac{\lambda}{n}} \left | I \right| _{\sigma}}  \approx \left ( \frac{h}{\ell \left (I \right)} \right )^{\kappa} \frac{1}{\left |I \right|^{1 - \frac{\lambda}{n}}} \sum\limits_{\beta} \left | u_{\beta} \right | \frac{\left | K _{\beta} \right | _{\sigma}}{\left |I \right|_{\sigma}} \, . 
\]
 Because we have fixed $\delta$ and $m$, which only depend on $\kappa$ and $n$, the doubling constants for $\sigma$ and $\omega$ and the Calder\'on-Zygmund data for $T$, \emph{we now absorb any constants depending on $\delta$ and $ m$ into the symbols $\lesssim, \approx$ and $\gtrsim$,} i.e.,
 \[
	\left | T_{\sigma} ^{\lambda} \varphi (x) \right | \approx  \frac{1}{\left |I \right|^{1 - \frac{\lambda}{n}}} \sum\limits_{\beta} \left | u_{\beta} \right | \frac{\left | K _{\beta} \right | _{\sigma}}{\left | I \right |_{\sigma}} \, .
\]
By doubling of $\sigma$, we have $\left | K _{\beta} \right | _{\sigma} \approx \left | I \right |_{\sigma}$. Combined with the estimate $\sum\limits_{\beta} \left | u _{\beta} \right | \approx 1$, we in fact obtain the nondegeneracy condition \eqref{eq:nondegen_original}. This completes the key modifications of the proof of Theorem \ref{main' Alpert}. The reader may check the other routine modifications.

\section{Concluding remarks}

For the orthonormal basis of Haar wavelets $\left\{
h_{Q}^{\mu, \gamma }\right\}  _{Q\in\mathcal{D}, \Gamma_{Q, n} ^{\mu}}$ in $\mathbb{R}^{n}$ for
\emph{doubling} measures $\mu$, we obtain the best theorem possible for
admissible truncations of the $\lambda$-fractional vector Riesz transform $\mathbf{R}^{\lambda, n}$ when $\lambda \neq 1$: for any dyadic
grid $\mathcal{D}$, the operator norm of $\mathbf{R}^{\lambda,n} _{\sigma}$ from $L^{2}\left(  \sigma\right)
$ to $L^{2}\left(  \omega\right)  $ is comparable to the sum of the two Haar
testing characteristics. However, our Theorems \ref{main''} and  \ref{thm:T1_to_Ta_all} show more: whenever a gradient elliptic operator satisfies a $T1$ theorem, then it can be improved to a Haar and Alpert testing theorem.

As Theorem \ref{Haar} shows, if the operator is gradient elliptic then the cube testing conditions and the $A_2$ condition are implied by the Haar testing condition when the measures are doubling. Regarding the $T1$ to $Th$ Theorem \ref{main''}, little is known  beyond gradient elliptic operators.
The $T1$ Theorem \ref{T1_L2_Riesz} is conjectured to hold the larger class of Stein elliptic Calder\'on-Zygmund operators.  It is unknown whether one can replace gradient elliptic operators with the larger class of Stein elliptic operators in Theorem \ref{Haar}.

Keeping with the assumption of gradient elliptic operators, the situation for arbitrary locally finite positive Borel measures
$\mu$ in $\mathbb{R}^{n}$ is quite abysmal. In this general
case, we don't even know if the operator norm is comparable to the Haar testing characteristics
when $T^{\lambda}$ is an admissible truncation of the simplest singular integral, the
\emph{Hilbert transform} $H$. Or even if the operator norm of
$H$ is comparable to the sum of the two Haar testing characteristics, but
taken now over \emph{all} dyadic grids $\mathcal{D}$! Notably, we do not yet have the technology to improve the known $T1$ theorem for the Hilbert transform into a Haar testing theorem. This remains an interesting open question.

We mention that almost nothing is known about the more abstract problem of identifying those triples $\left(  T,\mathcal{O}%
_{1},\mathcal{O}_{2}\right)  $ of bounded operators $T$ (or collections of
such) from a Hilbert space $\mathcal{H}_{1}$ to a Hilbert space $\mathcal{H}_{2}$, and orthonormal bases
$\mathcal{O}_{1}$ and $\mathcal{O}_{2}$ in these Hilbert spaces, respectively,
such that the operator norm of (each) $T$ is (uniformly) comparable to the sum
of the testing characteristics for $T$ associated to the orthonormal bases
$\mathcal{O}_{1}$ and $\mathcal{O}_{2}$.

Finally, one can consider the Banach space analogues of the above questions, where one replaces orthonormal bases with frames. If we consider a Calder\'on-Zygmund operator and its boundedness from $L^p (\sigma)$ to $L^p (\omega)$, one may for instance take the Haar wavelets $\{h_{I} ^{\sigma} \}_{I \in \mathcal{D}}$, $\{h_{I} ^{\omega} \}_{I \in \mathcal{D}}$ to be the frames associated to $L^p (\sigma)$ and $L^p (\omega)$, respectively. Similarly, weighted Alpert wavelets form a frame in weighted $L^p$ \cite{SaWi}, and thus the above questions can all be formulated for weighted Alpert wavelets as well.

\end{document}